\documentclass[aop,preprint,usenames,dvipsnames]{imsart} 
\usepackage{amsthm,amsmath,amsfonts,amssymb}
\usepackage{mathdots}
\usepackage{graphicx}
\usepackage{bm}
\usepackage{mathtools}
\usepackage{type1cm}
\usepackage{latexsym}
\usepackage{mathrsfs}
\usepackage{dsfont}
\usepackage{bbm}
\usepackage{soul}
\usepackage{epic}
\usepackage{tikz}
\usetikzlibrary{automata, positioning}
\usepackage{enumitem}
\usepackage[normalem]{ulem}
\usepackage{xcolor}
\usepackage[top=3.cm, bottom=3cm, outer=3.5cm, inner=3cm]{geometry}

\RequirePackage[colorlinks,citecolor=blue,urlcolor=blue]{hyperref}

\startlocaldefs

\newtheorem{Theorem}{Theorem}[section]
\newtheorem{Lemma}[Theorem]{Lemma}

\newtheorem{Prop}[Theorem]{Proposition}
\newtheorem{Rem}[Theorem]{Remark}

\def\ge{\geqslant}
\def\le{\leqslant}

\def\cC{\mathscr{C}}

\def\cE{\mathscr{E}}
\def\cF{\mathscr{F}}

\def\cL{\mathscr{L}}

\def\cS{\mathscr{S}}

\def\cX{\mathscr{X}}

\def\fX{\mathfrak{X}}

\def\C{\mathbb{C}}
\def\bbD{\mathbb{D}}
\def\Erw{\mathbb{E}}
\def\F{\mathbb{F}}

\def\N{\mathbb{N}}

\def\Prob{\mathbb{P}}

\def\R{\mathbb{R}}

\def\Z{\mathbb{Z}}

\def\sfp{\mathsf{p}}  \def\sfP{\mathsf{P}}

\def\sfm{\mathsf{m}}

\def\eps{\varepsilon}
\def\vph{\varphi}
\def\vth{\vartheta}

\def\ups{\upsilon}

\def\1{\mathbf{1}}
\def\3{{\ss}}

\def\wh{\widehat}

\def\ovl{\overline}

\def\sign{\textsl{sign}}

\def\diag{\textsl{diag}}
\def\bfi{\mathbf{i}}

\def\IRg{\R_{\scriptscriptstyle >}}
\def\IRge{\R_{\scriptscriptstyle\geqslant}}

\def\IRle{\R_{\scriptscriptstyle\leqslant}}

\def\dualxi{\wh{\xi}}

\def\dualups{\wh{\ups}}
\def\dualzeta{\wh{\zeta}}
\def\dualPsi{\wh{\Psi}}
\def\dualP{\wh{P}}
\def\dualsfP{\wh{\sfP}}
\def\dualp{\wh{p}}
\def\dualsfp{\wh{\sfp}}
\def\dualu{\wh{u}}
\def\dualv{\wh{v}}
\def\dualA{\wh{A}}
\def\dualB{\wh{B}}

\def\dualD{\wh{D}}

\def\dualF{\wh{F}}
\def\dualPi{\wh{\itPi}}
\def\dProb{\wh{\Prob}}
\def\dErw{\wh{\Erw}}

\def\dualR{\wh{R}}
\def\dualS{\wh{S}}

\def\dualXi{\wh{\Xi}}
\def\dualcF{\wh{\cF}}

\newcommand{\itPi}{\mathit{\Pi}}

\newcommand{\e}{{\mathrm{e}}}

\def\ALIFS{\textsf{ALIFS}}

\def\GOU{\textsf{GOU}}

\def\MAP{\textsf{MAP}}
\def\MMLIFS{\textsf{MMLIFS}}
\def\MMOU{\textsf{MMOU}}
\def\MMGOU{\textsf{MMGOU}}
\def\MRW{\textsf{MRW}}

\def\Cplus{C^{\scriptscriptstyle +}}
\def\Ciplus{C_{i}^{\scriptscriptstyle +}}
\def\Cminus{C^{\scriptscriptstyle -}}
\def\Ciminus{C_{i}^{\scriptscriptstyle -}}

\def\cblue{\color{blue}}

\def\cblack{\color{black}}

\newcommand{\ZZ}{\mathbb{Z}}
\newcommand{\CC}{\mathbb{C}}
\newcommand{\FF}{\mathbb{F}}
\newcommand{\diff}{\operatorname{d}\hspace{-1pt}}

\allowdisplaybreaks[4]

\begin{document}

\begin{frontmatter}

\title{Tail behavior of Markov-modulated generalized Ornstein-Uhlenbeck processes}
\runtitle{Tails of Markov-modulated generalized Ornstein-Uhlenbeck processes}


\begin{aug}
 \author{\fnms{Gerold}~\snm{Alsmeyer}\thanksref{t1}\corref{}\ead[label=e1]{gerolda@math.uni-muenster.de}}
  \and
  \author{\fnms{Anita}~\snm{Behme}
  \corref{}\ead[label=e2]{anita.behme@tu-dresden.de}}
  \runauthor{G.~Alsmeyer and A.~Behme}
  \affiliation{University of M\"unster, Technical University of Dresden}
\thankstext{t1}{G.~Alsmeyer was partially funded by the Deutsche Forschungsgemeinschaft (DFG) under Germany's Excellence Strategy EXC 2044--390685587, Mathematics M\"unster: Dynamics--Geometry--Structure.}

\address{G.~Alsmeyer\\Inst.~Math.~Stochastics\\ Department
of Mathematics\\ and Computer Science\\ University of M\"unster\\ Orl\'eans-Ring 10\\ D-48149 M\"unster, Germany\\
          \printead{e1}\\
          }

\address{A.~Behme\\Inst.~Math.~Stochastics \\
\& \\
Center for Scalable Data Analytics \\and Artificial Intelligence (ScaDS.AI) \\
Technical University of Dresden\\
D-01062 Dresden, Germany\\
           \printead{e2}}
\end{aug}

\begin{abstract}
We study the tail behavior of Markov-modulated generalized Ornstein-Uhlenbeck processes -- that is, solutions to Langevin-type stochastic differential equations driven by a background continuous-time Markov chain. To this end, we consider a sequence of Markov modulated 
random affine functions $ \Psi_{n} : \mathbb{R} \to \mathbb{R} $, $ n \in \mathbb{N} $, and the associated iterated function system defined recursively by $ X_0^x := x $ and $ X_{n}^x := \Psi_{n-1}(X_{n-1}^x) $ for $ x \in \mathbb{R} $, $n \in \mathbb{N}$. We analyze the tail behavior of the stationary distribution of such a Markov chain using tools from Markov renewal theory. Our approach extends Goldie's implicit renewal theory~\cite{Goldie:91} and can be seen as an adaptation of Kesten's work on products of random matrices~\cite{Kesten:73} to the one-dimensional setting of random affine function systems. These results have applications in diverse areas of applied probability, including queueing theory, econometrics, mathematical finance, and population dynamics.
\end{abstract}




\end{frontmatter}


\section{Introduction}

Langevin-type stochastic differential equations (SDEs) of the form
\begin{align}\label{MMGOUSDE}
	\diff V_{t}\ =\ V_{t-}\,\diff U_{t}\,+\,\diff L_{t},\qquad t\ge 0, 
\end{align}
are widely studied in the applied stochastics literature because they cover a large range of common models in various areas. Assuming that $(U,L)$ is a bivariate L\'evy process on $\R^{2}$ such that $\Delta U> -1$, i.e. $U$ a.s. has no jumps of size less or equal to $-1$, the solution to \eqref{MMGOUSDE} is known as a \emph{generalized Ornstein-Uhlenbeck (\GOU) process}. This type of process is a generalization of the classical Ornstein-Uhlenbeck process which, historically, has been introduced by the eponymous authors as a model in physics in the article \cite{OrnsteinUhlenbeck:1930} from 1930. In financial mathematics, special instances of \GOU\ processes include the classical Vasicek interest rate model \cite{Vasicek:77} and several prominent volatility models \cite{BarndorffShephard:01, CKluLM:04, MallerMuellerSzimayer}. In the actuarial sciences, \GOU\ processes are used as risk models with investment structure  \cite{Paulsen:93, GjessingPaulsen:97}. Recently, Langevin-type SDEs have also been proposed in machine learning as a continuous-time approximation of stochastic gradient descent (SGD) algorithms \cite{Mandtetal:16, Umutetal:19}.

\vspace{.1cm}
In all these applications it is of paramount importance to be able to quantify the tail behavior of the resulting process and/or its stationary distribution. Therefore, this problem has been investigated in various special cases \cite{CKluLM:06, GjessingPaulsen:97} as well as in the general L\'evy-driven setting \cite{BehmeAAP:11, AristaRivero:23}, where stationary solutions to \eqref{MMGOUSDE} are also known as \emph{exponential functionals}.

\vspace{.1cm}
A key tool to tackle this question used in these works is that solutions to \eqref{MMGOUSDE} appear as so-called \emph{perpetuities} after discretization, that is, solutions of random affine recursions of the form
\begin{equation} \label{eq-basicrre} 
	R_{n+1}\,=\,A_{n} R_{n}\,+\,B_{n}, \quad n\in \N,
\end{equation} 
for an i.i.d. sequence $(A_{n},B_{n})_{n\in\N}$. Tails of stationary solutions to \eqref{eq-basicrre} in the i.i.d.~case have been considered in the seminal works by Kesten \cite{Kesten:73} and Goldie \cite{Goldie:91}, who proved that - under suitable conditions - perpetuities in the stationary regime exhibit Pareto-like tails, which is also referred to as power-law behavior. As shown in \cite{LindnerMaller05} and \cite{BehmeAAP:11}, this behavior carries over to the continuous-time setting and - as stated above - it is of significant interest in applications. In the afore-mentioned actuarial context of a ruin model with investment, for example, the polynomial decay rates of certain exponential functionals imply slow decay of the ruin probability \cite[Prop. 3.1]{GjessingPaulsen:97} as function of the initial capital.

\vspace{.1cm}
In the present work, the case when $(U_{t},L_{t})_{t\ge 0}$ is modulated by a continuous-time Markov process $(J_{t})_{t\ge 0}$ will be addressed regarding the question how the modulation affects the tail behavior of a stationary solution to \eqref{MMGOUSDE}.
This means that $(J_{t},(U_{t},L_{t}))_{t\ge 0}$ forms a \emph{Markov-additive process (\MAP)} with additive component $(U_{t},L_{t})_{t\ge 0}$ in $\R^2$ such that again $\Delta U> -1$, 
and driving process $(J_{t})_{t\ge 0}$, which is here supposed to have finite state-space~$\cS$. 

Generally speaking, Markov modulation is a standard tool in many application areas to achieve greater flexibility in modeling. In the financial context Markov switching models have been used for about 40 years \cite{Hamilton:89,Hamiltonbook} by now, with Markov switching interest rates as well as Markov switching volatility models still being active topics of research  \cite{BenSaida:15, FinkCzado:17, BehmeMMCOGARCH:25}. Likewise, Markov-modulated risk models first appeared in the 1980s \cite{Janssen:80, Reinhard:84, AsmussenCKlu:94}, and Markov-modulated risk models with investments that solve an SDE of type \eqref{MMGOUSDE} have now been studied for about 15 years \cite{BaeuerleKoetter:07, SalahMorales, Ramsden:17, BSidMMGOU}. In particular, and in analogy to the non-modulated case \cite[Thm. 3.2]{Paulsen:93}, it is shown in \cite[Thm. 4.2]{BSidMMGOU} that the ruin probability of a Markov-modulated risk model with investment solving \eqref{MMGOUSDE} can be expressed in terms of a Markov-modulated exponential functional. Thus again the decay rates of exponential functionals determine the decay rates of ruin probabilities in the considered models. 
In the rather young and extremely active area of machine learning, first studies using Markov switching versions of SGD have appeared in \cite{SunSunYin:18, Evan:23}, and \eqref{MMGOUSDE} with $|\cS|>1$ may serve as an approximation for such SGD algorithms in the future.

\vspace{.1cm}
In the Markov-modulated setting, diffusion-type solutions to \eqref{MMGOUSDE} and those with jumps are called \emph{Markov-modulated Ornstein-Uhlenbeck (\MMOU) processes} and \emph{Markov-modula\-ted generalized Ornstein-Uhlenbeck (\MMGOU) processes}, respectively. For \MMOU\ processes, $U$ is supposed to be a Markov-modulated drift process and $L$ a Markov-modulated Brownian motion with drift. This class was studied e.g. in \cite{Bardetetal:10, LindskogMajumder:20, Huangetal:16}, the first two references regarding stationarity, ergodicity and also tail behavior, the last one regarding moments. There is also recent work on \MMGOU\ processes, see \cite{BSidMMGOU, BDTS}, but none of these consider tail behavior. On the other hand, the occurrence of jumps is clearly a necessary feature in applications, and the present work therefore aims to fill this gap.

\vspace{.1cm}
In principle, the results by Goldie and Kesten are still applicable in the Markov-modulated setting: Considering the return times $\tau_{n}(j), n\in\N$ of the Markov chain $J$ to an arbitrary state $j\in\cS$, one can easily check that any solution to \eqref{MMGOUSDE} fulfills the random recursive equation 
\begin{equation}\label{eq:MMGOUdiscretecoarse} V_{\tau_{n}(j)}\,=\, A_{n} V_{\tau_{n-1}(j)}\,+\, B_{n}, \end{equation}
with i.i.d.~coefficients $(A_{n},B_{n})$, $n\in\N$. This fact has been applied in \cite{LindskogMajumder:20} for the diffusion case. However, subsampling at return times does not provide a deeper insight into the contributions of different regimes to the tail behavior: The resulting conditions for Pareto-tails depend on integrals of the arising \MAP s over a cycle marked by two consecutive return epochs, which typically can hardly be re-expressed in terms of the characteristics of the given process. Therefore, a better understanding of the tail behavior of a stationary solution to \eqref{MMGOUSDE} calls for an extension of the Kesten-Goldie implicit renewal theorem \cite{Kesten:73,Goldie:91} to random affine recursions in a Markovian environment. Such extensions have so far only been considered and applied in some special cases. We mention \cite{Roitershtein:07, Saporta:05,Bardetetal:10} and especially the recent article \cite{AlsBroBur:23}, where implicit renewal theory for a two-state Markovian environment is developed and then used to determine the stationary tail behavior of so-called asymptotically linear iterated function systems (\ALIFS) on the line. The latter are characterized by being asymptotically equivalent to possibly distinct random affine functions at the two ends of the real line, thus at $\pm\infty$. Motivated by the fact that an environment of only two states is too restrictive for our purposes, this work provides an implicit renewal theorem for random affine recursions in a general finite Markovian environment, stated as Theorem \ref{thm:main_{1}}. We apply Theorem~\ref{thm:main_{1}} to a finer discretization of \eqref{MMGOUSDE}, as compared to the coarser scheme in \eqref{eq:MMGOUdiscretecoarse}, which yields precise conditions for the emergence of Pareto-like tails in Markov-modulated exponential functionals. These results will be presented in Theorem~\ref{thm:main_2}. The proof of Theorem~\ref{thm:main_{1}}, in turn, builds on and extends techniques developed in~\cite{AlsBroBur:23}.

\section{Main results}\label{sec:maincont}

In this section, we state our main result, Theorem~\ref{thm:main_2}, on the tail behavior of stationary solutions to the SDE \eqref{MMGOUSDE} for a \MAP\ $(J, (U, L))$. As indicated earlier, the proof is based on a discretization argument combined with Theorem~\ref{thm:main_{1}}. The proof of Theorem~\ref{thm:main_2} is therefore deferred to the final Section~\ref{sec:proof thm main_2}.

\vspace{.1cm}
\emph{Setup.} Let $\cS$ be a finite set and assume we are given a filtered probability space $(\Omega,\cF,\F,\Prob)$, where $\F = (\cF_t)_{t \ge 0}$ is a filtration of sub-$\sigma$-fields of $\cF$ satisfying the usual conditions of completeness and right-continuity. Let $(J,X) = (J_t, X_t)_{t \ge 0}$ be an $\F$-adapted Markov process taking values in $\cS \times \R^d$, with $d \ge 1$. We define $\Prob_j(\cdot) := \Prob(\cdot\,|\,J_0 = j)$, as usual, and let $\Erw_j[\cdot]$ denote the corresponding expectation operator. Then $(J,X)$ is called a \emph{($d$-dimensional) Markov-additive process (\MAP) with respect to $\F$} if
\begin{align}\label{MAPdefinition}
	\Erw\left[f(X_{s+t}-X_s)g(J_{s+t})|\cF_s\right] \ =\ \Erw_{J_s}\left[f(X_t - X_0)g(J_t)\right]
\end{align}
for all $s, t \ge 0$, and all bounded measurable functions $f : \R^d \to \R$, $g : \cS\to\R$. Since $\cS$ is finite, it follows in this case that a family of independent $\R^d$-valued L\'evy processes $\{X^{(j)} : j \in \cS\}$ can be defined such that, whenever $J_t = j$ on an arbitrary interval $(t_{1}, t_2)$, the increments $(X_t - X_{t_{1}})_{t_{1} < t < t_2}$ of the additive component of $(J,X)$ are distributed as $(X_t^{(j)} - X_{t_{1}}^{(j)})_{t_{1} < t < t_2}$; see \cite[Chapter~XI.2]{Asmussen:03}.\\ 
Whenever the continuous-time background process $J$ jumps at time $T_{n}$ from state $i$ to state $j$, it induces an additional jump $\Delta X_{T_{n}} =: Z_{X,n}^{ij}$ in $X$, the distribution of which, denoted by $G_{\!X}^{ij}$, depends only on the pair $(i,j)$ -- not on the jump time or index -- and is independent of all other randomness.\\ 
Given any c\`adl\`ag process $Y$, we denote by $Y_{t-}$ the left-hand limit of $Y$ at time $t > 0$ (with the convention $Y_{0-} := Y_0$), and define the jump size at time $t$ by $\Delta Y_t := Y_t - Y_{t-}$. We assume throughout that $X_0 = 0$ almost surely for any \MAP. For further background on \MAP s, we refer to \cite{Asmussen:03}, \cite[Appendix]{DDK2017}, or \cite{BDTS}.

\vspace{.3cm}
Given a bivariate \MAP\ $(J,(\zeta,\eta)) = (J_{t}, (\zeta_{t}, \eta_{t}))_{t \ge 0}$, the \emph{Markov-modulated generalized Ornstein-Uhlenbeck (\MMGOU) process driven by $(J,(\zeta,\eta))$} has been introduced in~\cite{BSidMMGOU} as the process
\begin{align}\label{MMGOUexplicit}
	V_{t}\ =\ \e^{-\zeta_{t}}\left( V_{0} + \int_{(0,t]} \e^{\zeta_{s-}}\,\diff \eta_s \right), \quad t \ge 0,
\end{align}
where the initial value $V_0$ and the process $(J_t, (\zeta_t, \eta_t))_{t \ge 0}$ are conditionally independent given $J_0$. It has further been shown there that $(V_t)_{t \ge 0}$ forms the unique solution to the SDE~\eqref{MMGOUSDE} for a bivariate \MAP\ $(J, (U, L)) = (J_t, (U_t, L_t))_{t \ge 0}$ that is uniquely determined by $(J, (\zeta, \eta))$ via 
\begin{equation}\label{eq-ULviaxieta}
	\begin{pmatrix}
		U_t \\ L_t
	\end{pmatrix}\ =\  
	\begin{pmatrix}
		-\zeta_t + \frac12 \int_{(0,t]} \sigma_\zeta^2(J_s)\,\diff s + \sum_{0<s\leq t} \left( \Delta \zeta_s + \e^{-\Delta \zeta_s} - 1 \right) \\[.3em]
		\eta_t -\int_{(0,t]} \sigma_{\zeta,\eta}(J_s)\,\diff s +\sum_{0<s\leq t} \left(\e^{-\Delta\zeta_s}-1\right) \Delta\eta_s
	\end{pmatrix}, \quad t\ge 0,
\end{equation}
see~\cite[Eq.\,(2.24)]{BSidMMGOU} or \cite[Eq.~(4.1)]{BDTS}. In particular, the resulting process satisfies $\Delta U > -1$.

\vspace{.1cm}
Under the assumption that the background driving Markov chain $J$ is ergodic with stationary distribution $\pi$, necessary and sufficient conditions for the existence of a stationary distribution of the \MMGOU\ process \eqref{MMGOUexplicit} are given in \cite[Thm. 3.3]{BSidMMGOU}. To state these, note that we write $\Prob_\pi[\cdot] := \sum_{i \in \cS} \pi_{i} \Prob_{i}[\cdot]$, as is standard.
Further, $(\wh{J}, (\wh{\zeta}, \wh{L}))$ denotes the dual of the Markov-additive process $(J, (\zeta, L))$, and we define $\wh{\Prob}_\pi [\cdot] := \sum_{i \in \cS} \pi_{i} \wh{\Prob}_{i} [\cdot]$ via $\wh{\Prob}_{i} [\cdot] := \Prob[\cdot\,|\,\wh{J}_0 = i]$, such that in particular
\begin{align*}
	&\text{the law of }\big(J_{(t-s)-},\, \zeta_{(t-s)-} - \zeta_t,\, \eta_{(t-s)-} - \eta_t \big)_{0 \leq s \leq t} \text{ under } \Prob_\pi \\
	&\hspace{3cm}\text{ is equal to the law of } \big(\wh{J}_s,\, \wh{\zeta}_s,\, \wh{\eta}_s \big)_{0 \leq s \leq t} \text{ under } \wh{\Prob}_\pi,
\end{align*}
see e.g.\ \cite[Appendix A.2]{DDK2017}. Then by \cite[Thm. 3.3]{BSidMMGOU}, the \MMGOU\ process admits a nontrivial stationary distribution if and only if the integral $\int_{(0,t]} \e^{\wh{\zeta}_{s-}}\,\diff \wh{L}_s$ converges in $\wh{\Prob}_\pi$-probability as $t \rightarrow \infty$ to some finite-valued random variable $V_{\infty}$. Moreover, in this case the stationary distribution of the \MMGOU\ process under $\Prob_\pi$ is uniquely determined as the distribution of the \emph{exponential functional}
\begin{equation} \label{eq_defVinfty}
	V_{\infty}\ =\ - \int_{(0,\infty)} \e^{\wh{\zeta}_{s-}}\,\diff \wh{L}_s\ =\ \wh{\Prob}_\pi\text{\,-}\lim_{t \to \infty} \left( - \int_{(0,t]} \e^{\wh{\zeta}_{s-}}\,\diff \wh{L}_s \right).
\end{equation}
Necessary and sufficient conditions for convergence of integrals of the form \eqref{eq_defVinfty} in terms of the characteristics of the appearing processes have been derived in \cite{BS_ExpFuncMAP}.
\\
In the case the convergence in \eqref{eq_defVinfty} fails, either no stationary distribution of the \MMGOU\ pro\-cess exists, or it is degenerate in the sense that there exists a sequence of constants $\{c_j,\, j \in \cS\}$ such that $V_t = c_{J_t}$ $\Prob_\pi$-a.s.\ for all $t \geq 0$, and $\pi_j = \Prob_\pi(V_t = c_j)$. As outlined in the proof of \cite[Thm. 3.3]{BSidMMGOU}, the latter behaviour appears if and only if $\wh{\Prob}_\pi$-almost surely 
$$ \int_{(0,t]} \e^{\wh{\zeta}_{s-}}\,\diff \wh{L}_s\ =\ c_{\wh{J}_0} + c_{\wh{J}_t}\e^{\wh{\zeta}_t}, $$
which in turn is equivalent to an explicit dependency of the driving processes of the form
\begin{equation} \label{eq-MMGOUdegenerate} L_t = -\int_{(0,t]} c_{J_s-} \diff U_s - \int_{(0,t]}\diff c_{J_s}, \quad \Prob_\pi\text{-a.s.},\end{equation}
see \cite[Prop. 4.7]{BS_ExpFuncMAP}.  We will exclude this case from our studies.

\vspace{.3cm}

In order to formulate our result on the tail behavior of the stationary distribution of the process $(V_t)_{t \geq 0}$, we denote the intensity matrix of $J$ by $Q = (q_{ij})_{i,j \in \cS}$. Further, we define
\begin{equation} \label{eq:def.psi.lap.tran}
	\Psi_\zeta(w)\ =\ \diag\left[\psi_j(w): j \in \cS\right] + Q^\top \circ\left(\Erw\left[\e^{w Z^{jk}_{\zeta,1}}\right]\right)_{j,k \in \cS}^\top\ \in\ \CC^{|\cS| \times |\cS|},  
\end{equation}
for all $w \in \CC$ such that the right-hand side exists. This quantity is known as the \emph{matrix exponent} of the \MAP\ $(J, \zeta)$; see, e.g., \cite[Prop. XI.2.2]{Asmussen:03} or \cite{DDK2017}.
Here and in the following, $\diag[a_j : j \in \cS]$ denotes the diagonal matrix with entries $a_j$ for $j = 1, \dots, |\cS|$; $\top$ denotes matrix transposition; and ``$\circ$'' denotes the elementwise (Hadamard) product of matrices. Moreover, $\psi_j(w) := \log \Erw[\e^{w \zeta^{(j)}_{1}}]$ is the Laplace exponent of the L\'evy process $\zeta^{(j)}$. As described in \cite[Sec. A.1]{DDK2017}, the matrix exponent determines the joint Laplace transform of the additive component $\zeta$ and the modulating chain $J$. Specifically,
\begin{equation} \label{eq:lap.tran.map}
	\Erw_{i}\Big[\e^{w \zeta_t} \mathds{1}_{\{J_t = j\}}\Big]\ =\ e_j^\top \e^{t \Psi_\zeta(w)} e_{i}, \quad \text{for all } i, j \in \cS,
\end{equation}
where $e_j$ denotes the $j$-th unit vector.
A closely related quantity that will also be important is the matrix $\Upsilon_\zeta(\theta)$, defined for $\theta \ge 0$ by
\begin{equation} \label{eq_upsilon}
	\Upsilon_\zeta^{ij}(\theta)\ :=\ \Erw_{i}\Big[\e^{-\theta \zeta_{T_{1}}} \mathds{1}_{\{J_{T_{1}} = j\}}\Big] 
	\ =\ q_{ij} \int_0^\infty e_j^\top \e^{t \Psi_\zeta(-\theta)} e_{i} \, \e^{-q_{ij} t} \,\,\diff t, \quad i, j \in \cS,
\end{equation}
where $T_{1}$ denotes the time of the first regime switch, i.e.,
$$
T_{1}\,:=\,\inf\{t > 0 : J_t \neq J_{t-}\}.
$$

Our main result on the tail behavior of the stationary distribution is the following. Its proof is provided in Section~\ref{sec:proof thm main_2}, and it builds upon an extension of Goldie's implicit renewal theorem to the Markov-modulated setting, which will be stated as Theorem~\ref{thm:main_{1}} in the next section.

\begin{Theorem}\label{thm:main_2}
	Let $(J,(\zeta,\eta))$ be a bivariate \MAP\ such that the modulating process $J$ has finite state space $\cS$ and is ergodic with stationary distribution $\pi$. Assume that $\eta\not\equiv 0$  and there exists no sequence $\{c_j,\,j\in\cS\}$ such that \eqref{eq-MMGOUdegenerate} holds. 
	If all Levy processes $\zeta^{(j)}$ are of finite variation with zero drift, additionally assume that either there exists $j\in\cS$ such that the jump distribution of $\zeta^{(j)}$ is nonlattice or that $Z_\zeta^{ij}$ is nonlattice for at least one pair $(i,j)\in\cS^2$ with $q_{ij}>0$. Suppose further that the following conditions hold:
	\begin{itemize}[leftmargin=1.2cm]\itemsep3pt
		\item[\textsf{(A1) }]
		There exists $\kappa > 0$ such that $\Upsilon_\zeta(\kappa)$ has dominant eigenvalue $1$ and is irreducible;
		
		\item[\textsf{(A2) }] 
		For all $i, j \in \cS$, the derivative
		$$
		\frac{\partial}{\partial \theta} \Upsilon_\zeta^{ij}(\theta)\Big|_{\theta=\kappa}\ <\ \infty;
		$$
		
		\item[\textsf{(A3) }] 
		$\Erw_\pi[|\eta_{1}|^{\kappa \vee 1}]\ <\ \infty$;
		
		\item[\textsf{(A4) }] 
		There exists $\varepsilon > 0$ such that for all $i, j \in \cS$ with $q_{ij} > 0$,
		$$
		\Erw\left[\e^{-((\kappa \vee 1) + \varepsilon) Z^{ij}_{\zeta,1}}\right]\ <\ \infty,
		\quad \text{and} \quad 
		\int_0^\infty \Erw\left[\sup_{s \le t} \e^{((\kappa \vee 1) + \varepsilon) |\zeta^{(j)}_s|} \right] \e^{q_{jj} t}\,\diff t\ <\ \infty.
		$$
	\end{itemize}
	Then there exists a nondegenerate random variable $V_0$, which is conditionally independent of $(J, (\zeta, \eta))$ given $J_0$, such that the \MMGOU\ process defined in \eqref{MMGOUexplicit}, driven by $(J, (\zeta, \eta))$ and initialized at $V_0$, is strictly stationary. Furthermore, there exist finite constants $C_{i}^+, C_{i}^-$ such that
	\begin{gather*}
		\lim_{x \to \infty} x^\kappa \, \Prob_{i}[V_0 > x]\,=\,C_{i}^+,
		\quad\text{and}\quad
		\lim_{x \to \infty} x^\kappa \, \Prob_{i}[V_0 < -x]\,=\,C_{i}^-,
	\end{gather*}
	with $C_i^+ + C_i^->0$  for all $i \in \cS$.
\end{Theorem}

\begin{Rem}\rm
	Let $u(\kappa)=(u_{i})_{i\in\cS}$ and $v(\kappa)=(v_{i})_{i\in\cS}$ denote the positive left and right eigenvectors corresponding to the leading eigenvalue of $\Upsilon_\zeta(\kappa)$, normalized by $u(\kappa)^{\top}v(\kappa)=1$. Let $(\dualv_{i}(\kappa))_{i\in\cS}$ be the vector of the reciprocals of the $v_{i}(\kappa)$, which turns out to play the role for the dual operator of $\Upsilon_\zeta(\kappa)$ that $v(\kappa)$ plays for $\Upsilon_\zeta(\kappa)$, see \eqref{eq:def vhat}. Then there exist nonnegative constants $K^{-}$ and $K^{+}$, independent of $i$, such that
	$$
	C_{i}^{+}\ =\ K^{+}\,\dualv_{i}
	\quad\text{and}\quad
	C_{i}^{-}\ =\ K^{-}\,\dualv_{i}
	\quad\text{for all }i\in\cS.
	$$
	We postpone a more explicit description of $K^{\pm}$ until Section~\ref{sec:proof thm main_{1}} -- see \eqref{eq:def Ciplus} and \eqref{eq:def Ciminus} for the precise formulas -- because their definition requires notation that will be introduced more naturally there.
\end{Rem}

\begin{Rem}\label{rem:DeltaU}\rm
	As shown in \cite{BSidMMGOU}, the SDE \eqref{MMGOUSDE} can also be solved in the more general case $\Delta U \neq -1$, with solution given by
	\begin{equation} \label{MMGOUgeneralized} 
		V_t\ =\ \cE(U)_t \left( V_0\,+\, \int_{(0,t]} \cE(U)_{s-}^{-1}\,\diff \eta_s \right), \quad t \geq 0,
	\end{equation}
	where $\cE(U)$ denotes the stochastic exponential of $U$, defined as the solution to the SDE $\diff\cE(U)_t = \cE(U)_{t-} \, \diff U_t$ with initial condition $\cE(U)_0 = 1$. The stochastic exponential of a semimartingale -- such as the additive component of a \MAP\ -- can be expressed explicitly; see, e.g., \cite[Thm. II.37]{Protterbook}. From this representation, it is immediate that any jump of $U$ of size less than $-1$ causes a sign change in $\cE(U)$. In terms of perpetuities (i.e., time discretizations of solutions to \eqref{MMGOUSDE}), this corresponds to random affine recursions of the form $R_{n} = A_{n} R_{n-1} + B_{n}$, $n\in\N$, with $\Prob(A_{n} < 0) > 0$. This setting is also covered by Theorem~\ref{thm:main_{1}} in the discrete-time case. The application to the process given by \eqref{MMGOUgeneralized} can be carried out in direct analogy with the proof of Theorem~\ref{thm:main_2}; however, we omit the details here in order to keep the presentation accessible.
\end{Rem}

\begin{Rem}\rm
	A related result on the tail behavior of the classical generalized Ornstein-Uhlenbeck process (corresponding to the special case $|\cS| = 1$) is given in \cite[Thm. 4.5]{LindnerMaller05}; see also \cite[Thm. 4.1]{BehmeAAP:11} for a generalization in the spirit of Remark~\ref{rem:DeltaU} and under refined assumptions. Applying Theorem~\ref{thm:main_2} in this setting recovers the known results, but under a slightly weaker condition on the moments of $\eta$: we only require assumption \textsf{(A3)}, whereas \cite{BehmeAAP:11} assumes $\Erw|\eta_{1}|^{(\kappa + \varepsilon) \vee 1}<\infty$. This relaxation is possible due to new moment bounds for stochastic integrals derived in \cite{BDTS}.
\end{Rem}

\section{An implicit renewal theorem in finite Markovian environment}\label{sec:maindiscrete}

The Kesten-Goldie implicit renewal theorem (\cite[Thm.\,5]{Kesten:73} and \cite[Thm. 2.3]{Goldie:91}) characterizes the stationary tail behavior of iterations of i.i.d.~random functions $\Psi_{n}(x)$ on the real line under suitable conditions, a key one being that $\Psi_{n}(x)$ is approximately linear as $|x| \to \infty$, in an appropriate sense. The purpose of this section is to extend the Kesten-Goldie result to the setting of real-valued random affine linear functions $A_{n} x + B_{n}$ evolving in a Markovian environment.

\vspace{.1cm}
\emph{Setup}. Let $\Xi=(\xi_{n})_{n\in\Z}$ be a doubly infinite stationary and irreducible Markov chain with finite state space $\cS$, transition matrix $P=(p_{ij})_{i,j\in\cS}$ and stationary law $\pi=(\pi_{i})_{i\in\cS}$. We consider a sequence of random affine linear maps 
$$ \Psi_{n}(x)\ =\  A_{n}x+B_{n},\quad n\in\Z,\ x\in\R, $$ 
such that $(A_{n}, B_{n})_{n\in\Z}$ is taking values in $\R^{2}$ and modulated by $\Xi$. This means that, conditioned upon $\xi_{k}=i_{k}$ for $k\in\Z$ and arbitrary $i_{0},i_{\pm 1},i_{\pm 2}\ldots\in\cS$,
\begin{itemize}[leftmargin=.6cm]\itemsep2pt
	\item the $(A_{n},B_{n})$ are conditionally independent,
	\item the conditional law of $(A_{n},B_{n})$ is homogeneous in $n$ and only dependent on $(i_{n-1},i_{n})$, i.e. 
	$$ \Prob\big[(A_{n},B_{n})\in\cdot\,|\xi_{n-1}=i_{n-1},\,\xi_{n}=i_{n}\big]\ =\ F_{i_{n-1}i_{n}} $$ 
	for all $n\in\Z$ and a stochastic kernel $F$ from $\cS^{2}$ to $\R^{2}$, where $F_{ij}$ is used as shorthand for $F((i,j),\cdot)$.
\end{itemize}

The \emph{Markov-modulated linear iterated function system (\MMLIFS)} associated with the  se\-quence $(\Psi_{n})_{n\ge 0}$ of random affine linear functions, and with driving chain $(\xi_{n})_{n\ge 0}$, is defined by $(\xi_{n},R_{n})_{n\ge 0}$, where
\begin{gather}\label{eq:MIFS}
	R_{n}\ :=\ \Psi_{n}(R_{n-1})\ =\ \Psi_{n}\cdots\Psi_{1}(R_{0})\quad\text{for }n=1,2,\ldots
\end{gather}
and $\Psi_{n}\cdots\Psi_{1}$ is used as shorthand for the composition $\Psi_{n}\circ\ldots\circ\Psi_{1}$. It constitutes a temporally homogeneous Markov chain having transition kernel
$$ \Gamma((i,r),\{j\}\times\cdot)\ :=\ p_{ij}\,\Prob\big[A_{1}r+B_{1}\in\cdot\,|\xi_{0}=i,\xi_{1}=j\big]\quad\text{a.s.} $$
for all $i,j\in\cS$ and $r\in\R$, provided that $R_{0}$ is measurable with respect to $\cF_{0}$, where
$$ \cF_{n}\,:=\,\sigma((\xi_{k},A_{k},B_{k});k\le n)\quad\text{for }n\in\Z $$
is the natural filtration induced by $(\xi,A,B)$.
Notice that the conditional law of $(\xi_{n},A_{n},B_{n})_{n\ge 1}$ given $(\xi_{0},R_{0})$ does only depend on $\xi_{0}$ and not on $R_{0}$, thus
\begin{equation*}
	\Prob\big[(\xi_{n},A_{n},B_{n})_{n\ge 1}\in\cdot\,|\xi_{0}=i,R_{0}=r\big]\ =\ \Prob_{i}\big[(\xi_{n},A_{n},B_{n})_{n\ge 1}\in\cdot\big],
\end{equation*}
where $\Prob_{i}$ is used as common for $\Prob[\,\cdot\,|\xi_{0}=i]$ for $i\in\cS$ and $\Prob_{\lambda}$ for $\sum_{i\in\cS}\lambda_{i}\Prob_{i}$ if $\lambda=(\lambda_{i})_{i\in\cS}$ is a probability distribution on $\cS$.

\vspace{.2cm}
The $R_{n}$ are also called \emph{forward iterations} when opposed to the corresponding sequence of \emph{backward iterations}
\begin{gather}\label{backward sequence}
	\Psi_{1}\cdots\Psi_{n}(R_{0})\ =\ \itPi_{n}R_{0}\ +\ \sum_{k=1}^{n}\itPi_{k-1}B_k,\quad n=1,2,\ldots,
	\shortintertext{where}
	\itPi_{0}\ :=\ 1\quad\text{and}\quad\itPi_{n}\ :=\ \begin{cases}
		\hfill A_{1} A_{2}\cdot\ldots\cdot A_{n}&\text{if }n=1,2,\ldots\\
		A_{0} A_{-1}\cdot\ldots\cdot A_{n+1}&\text{if }n=-1,-2,\ldots
	\end{cases}\label{eq:def Pi(n)}
\end{gather}
If $R_{0}=0$, these iterations take the form
\begin{equation}\label{sequence}
	\Psi_{1}\cdots\Psi_{n}(0)\ =\ \sum_{k=1}^{n}\itPi_{k-1}B_{k},\quad n=1,2,\ldots
\end{equation}
with obvious limiting random variable as $n\to\infty$ (if it exists)
\begin{equation}\label{eq:perp}
	\dualR_{0}\ :=\ \sum_{k\ge 1}\itPi_{k-1}B_{k}.
\end{equation} 
The latter is called perpetuity due to its interpretation in insurance and  finance as a dis\-counted sum of perpetual payments. In the case when the $(A_{n},B_{n})$ are i.i.d., Goldie and Maller \cite[Thm.\,2.1]{GolMal:00} established necessary and sufficient conditions for $\dualR_{0}$ to be a.s.~finite and its law to be the unique stationary law of the Markov chain $(R_{n})_{n\ge 0}$. In the present situation of Markov modulated coefficients, corresponding results have been obtained in \cite{AlsBuck:17a}, but are complicated by the fact that in general the stationary law of $(R_{n})_{n\ge 0}$ is no longer the law of $\dualR_{0}$ but of a dual counterpart. In fact, the following has been shown by Brandt \cite{Brandt:86}: The choice
\begin{gather}\label{eq:stationary Rzero}
	R_{0}\ :=\ \sum_{k\ge 0}\itPi_{-k}B_{-k}\ =\ B_{0}+A_{0}B_{-1}+A_{0}A_{-1}B_{-2}+\ldots
\end{gather}
renders stationarity of $(\xi_{n},R_{n})_{n\ge 0}$ if the right-hand sum in \eqref{eq:stationary Rzero} is a.s.~finite. In this case, a doubly infinite $(\cF_{n})_{n\in\Z}$-adapted stationary extension of the latter sequence can be explicitly defined in terms of the given stationary sequence $(\xi_{n},A_{n},B_{n})_{n\in\Z}$, namely
\begin{equation}
	R_{n}\ :=\ \sum_{k\ge 0}A_{n}\cdots A_{n-k+1}B_{n-k}\quad\text{for }n\in\Z.\label{eq:stationary Rn}
\end{equation}
It is this stationary representation we will work with hereafter when studying the tail behavior of the stationary law of $(R_{n})_{n\ge 0}$. Let us point out that $R_{0}$, when formally defined by \eqref{eq:stationary Rzero}, is also a backward iteration limit and thus a perpetuity, but for the time-reversal $(\xi_{-n},A_{-n},B_{-n})_{n\ge 0}$, namely
\begin{equation}
	R_{0}\ =\ \Psi_{0} \cdots\Psi_{-n}(R_{-n})\quad\text{for all }n\in\N_{0}.\label{eq2:stationary Rn}
\end{equation}
As a consequence,  the results from \cite{AlsBuck:17a} can be used to directly infer the a.s.~finiteness of $R_{0}$ under the conditions stated in our main results below.

\vspace{.2cm}
\emph{Nondegeneracy}. In the case of i.i.d.~$(A_{n},B_{n})$, it is well-known that the law of $R_{0}$ is nondegenerate iff
\begin{equation}\label{nondegenerate iid}
	\Prob[A_{1}c+B_{1}=c\,]\,<\,1\quad\text{for all }c\in\R,
\end{equation}
see \cite[Rem.\,2.3]{GolMal:00}. For the Markov-modulated case, the corresponding condition has been provided in \cite[Subsect.\,2.4]{AlsBuck:17a}, namely
\begin{equation}\label{nondegenerate MMLIFS}
	\Prob_{\pi}[A_{1}c_{\xi_{1}}+B_{1}=c_{\xi_{0}}\,]\,<\,1\quad\text{for any choice }c_{i}\in\R,\,i\in\cS.
\end{equation}
In the following, the \MMLIFS\,$(\xi_{n},R_{n})_{n\ge 0}$ associated with $(A_{n},B_{n})$ is called nondegenerate if this condition holds.

\vspace{.2cm}
We note that the definitions of $\Prob_{i}:=\Prob[\cdot\,|\,\xi_{0}=i]$ and 
$\Prob_{\lambda}:=\sum_{i\in\cS}\lambda_{i}\Prob_{i}$ for any probability law 
$\lambda=(\lambda_{i})_{i\in\cS}$ on $\cS$ remain in force for the doubly infinite sequence 
$(\xi_{n},A_{n},B_{n})_{n\in\Z}$. Moreover, we emphasize that stationarity under $\Prob_{\pi}$ holds 
not only for $(\xi_{n})_{n\in\Z}$ but also for the extended sequences 
$(\xi_{n},A_{n},B_{n})_{n\in\Z}$ and $(\xi_{n},R_{n})_{n\in\Z}$.

\vspace{.2cm}
The following result is the announced extension of the Kesten--Goldie implicit renewal theorem 
to \MMLIFS. In the form stated, the i.i.d.\ case for $(A_{n},B_{n})$ appears as 
Theorems 2.4.4 and 2.4.7 in the monograph by Buraczewski et al.\ \cite{BurDamMik:16}.

\begin{Theorem}\label{thm:main_{1}}
	Let $(\xi_{n},A_{n},B_{n})_{n\in\Z}$ be a doubly infinite stationary sequence as introduced above, 
	and $(\xi_{n},R_{n})_{n\ge 0}$ the \MMLIFS\ associated with $\Psi_{n}(x)=A_{n}x+B_{n}$ for 
	$n\in\N_{0}$. Further assume
	\begin{itemize}[leftmargin=1.2cm]\itemsep3pt
		\item[\textsf{(B1)}] There exists $\kappa>0$ such that 
		$\sfP(\kappa):=\big(\Erw_{i}[\,|A_{1}|^{\kappa}\1_{\{\xi_{1}=j\}}]\big)_{i,j\in\cS}$ 
		has dominant eigenvalue $\rho(\kappa)=1$.
		\item[\textsf{(B2)}] $\Erw|A_{1}|^{\kappa}\log|A_{1}|<\infty$ and $\Erw|B_{1}|^{\kappa}<\infty$.
		\item[\textsf{(B3)}] 
		$\Prob\big[\log|A_{1}|-a_{\xi_{1}}+a_{\xi_{0}}\in d\Z\,\big|\,A_{1}\ne 0\big]<1$ 
		for any $d>0$ and any choice of real numbers $\{a_{i}:i\in\cS\}$.
		\item[\textsf{(B4)}] $\Prob[B_{1}=0]<1$, hence 
		$\min_{i,j\in\cS}F_{ij}(\IRg\times\{0\})<1$.
	\end{itemize}
	Then $-\infty\le\Erw\log|A_{1}|<0$, and $(R_{n})_{n\ge 0}$ has a unique stationary law and a 
	doubly infinite extension as defined in \eqref{eq:stationary Rn}. Moreover, there exist nonnegative 
	constants $\Ciminus$ and $\Ciplus$ (explicitly defined in Section \ref{sec:proof thm main_{1}}) such 
	that the following assertions hold:
	
	\vspace{.2cm}\noindent
	(a) If $\Prob[A_{1}\ge 0]=1$, then
	\begin{equation}\label{eq1:left-right tail nu}
		\lim_{t\to\infty}t^{\kappa}\,\Prob_{i}[R_{0}>t]\ =\ \Ciplus,
		\qquad
		\lim_{t\to\infty}t^{\kappa}\,\Prob_{i}[R_{0}<-t]\ =\ \Ciminus
		\quad\text{for all }i\in\cS.
	\end{equation}
	
	\noindent
	(b) If $\Prob[A_{1}<0]>0$, then
	\begin{equation}\label{eq2:left-right tail nu}
		\lim_{t\to\infty}t^{\kappa}\,\Prob_{i}[R_{0}>t]\ =\ \lim_{t\to\infty}t^{\kappa}\,\Prob_{i}[R_{0}<-t]\ 
		=\ \frac{\Ciminus+\Ciplus}{2}\quad\text{for all }i\in\cS.
	\end{equation}
	In both cases, $\Ciplus+\Ciminus>0$ if $(\xi_{n},R_{n})_{n\ge 0}$ is nondegenerate.
\end{Theorem}

\begin{Rem}\label{rem:Pkappa irreducible}\rm
	The matrix $\sfP(\kappa)$ defined in \textsf{(B1)} is irreducible. This follows from the irreducibility of the Markov chain $\Xi$ with transition matrix $P$, as assumed at the beginning of this section. Since the state space $\cS$ is finite and assumption \textsf{(B1)} ensures that $A_{1}$ does not vanish almost surely, there exists $\varepsilon > 0$ such that $\sfP(\kappa) \ge \varepsilon P$. Consequently, for all $n\in\N$, we have $\sfP(\kappa)^n \ge \varepsilon^n P^n$, which in turn implies the irreducibility of $\sfP(\kappa)$.
\end{Rem}

\begin{Rem}\label{rem:Lipschitz extension}\rm
	Let $\cC^{*}(\R)$ denote the space of bounded Lipschitz functions $\phi$ on $\R$ that vanish in a neighborhood of the origin, and let $\cC_{-}^{*}(\R)$ and $\cC_{+}^{*}(\R)$ be the subspaces consisting of functions in $\cC^{*}(\R)$ that additionally vanish on $\IRge$ and $\IRle$, respectively.\\ 
	Since every $\phi \in \cC^{*}(\R)$ is almost everywhere differentiable with bounded derivative $\phi'$ that also vanishes on a neighborhood of the origin, say $[-\eps, \eps]$, integration by parts  together with an application of \eqref{eq1:left-right tail nu}
	yields
	\begin{align}\label{eq1:Lipschitz extension}
		\begin{split}
			\lim_{t \to \infty} t^{\kappa}\,\Erw_{i}\phi(t^{-1}R_{0})\ 
			&=\ \kappa \int_{0}^{\infty}\frac{\Ciminus \phi(-x) + \Ciplus \phi(x)}{x^{\kappa+1}} \, \diff x.
		\end{split}
	\end{align}
\end{Rem}

\section{Prerequisites}

This section collects the main prerequisites required for the proof of Theorem~\ref{thm:main_{1}}.

\subsection{Time-reversal and duality}

The time-reversal and dual of the trivariate sequence $(\xi_{n}, A_{n}, B_{n})_{n \ge 0}$ is defined by
$$
(\dualxi_{n}, \dualA_{n}, \dualB_{n})_{n \in \Z}\,:=\,(\xi_{-n}, A_{-n}, B_{-n})_{n \in \Z},
$$
which again forms a stationary process, with $(\dualA_{n}, \dualB_{n})_{n \in \Z}$ modulated by the Markov chain $\dualXi := (\dualxi_{n})_{n \ge 0}$. However, the transition mechanism generally differs. Specifically:

\begin{itemize}[leftmargin=.6cm]\itemsep2pt
	\item $\dualXi := (\dualxi_{n})_{n \ge 0}$ is the time-reversal of $\Xi$, with transition probabilities $\dualp_{ij}:=\pi_{j}p_{ji}/\pi_{i}$ and transition matrix $\dualP = (\dualp_{ij})_{i,j \in \cS}$.
	\item The conditional law of $(\dualA_{n}, \dualB_{n})$ given $\dualXi$ depends on $(\dualxi_{n}, \dualxi_{n+1})$ rather than $(\dualxi_{n-1}, \dualxi_{n})$, and satisfies
	$$ \dualF_{ij}\,:=\,\Prob\big[(\dualA_{n}, \dualB_{n}) \in \cdot \,\big|\, \dualxi_{n} = i,\, \dualxi_{n+1} = j\big]\,=\, F_{ji}, $$
	for all $n \ge 0$ and $i,j \in \cS$, where $F$ is the stochastic kernel introduced in the previous section.
\end{itemize}
This shows that, strictly speaking, duality holds between the sequences $(\xi_{n}, A_{n+1}, B_{n+1})_{n \ge 0}$ and $(\dualxi_{n}, \dualA_{n}, \dualB_{n})_{n \ge 0}$, and that they are equal in law if and only if the driving chain $\Xi$ is time-reversible and the modulation is symmetric, in the sense that $F_{ij} = F_{ji}$ for all $i,j \in \cS$.

\vspace{.1cm}

Define $\dualPsi_{n}(x) := \dualA_{n-1} x + \dualB_{n-1}$ for $n \in \Z$. Then the dual of the \MMLIFS\ $(\xi_{n}, R_{n})_{n \ge 0}$ is the sequence $(\dualxi_{n}, \dualR_{n})_{n \ge 0}$ with
\begin{equation} \label{ifsf}
	\dualR_{n}\ :=\ \dualPsi_{n}(\dualR_{n-1})\ =\ \dualPsi_{n} \cdots \dualPsi_{1}(\dualR_0), \quad n = 1, 2, \ldots
\end{equation}
The corresponding backward iterations are given by
\begin{equation*}
	\dualPsi_{1} \cdots \dualPsi_{n}(0)\ =\ \sum_{k=0}^{n-1} \dualPi_k \dualB_k, \quad n = 1, 2, \ldots,
\end{equation*}
where $\dualPi_0 := 1$ and $\dualPi_{n} := \prod_{k=0}^{n-1} \dualA_k = \prod_{k=0}^{n-1} A_{-k}$ for $n \ge 1$. If these iterations converge, then their limit clearly equals $R_0$ as defined in \eqref{eq:stationary Rzero}, and the law of $R_0$ coincides with the stationary distribution of the process $(\dualR_{n})_{n \ge 0}$. Moreover, by an application of Brandt's result, the doubly infinite sequence $(\dualxi_{n}, \dualR_{n})_{n \in \Z}$ is stationary where 
\begin{equation*}
	\dualR_{n}\ :=\ \sum_{k \ge 0} \dualA_{n} \cdots \dualA_{n-k+1}\dualB_{n-k}\ =\ \sum_{k\ge 0} A_{-n} \cdots A_{-n+ k-1} B_{-n+k}
\end{equation*}
for all $n \in \Z$. Note that the definition of $\dualR_0$ coincides with \eqref{eq:perp}. 

Summarizing, let $\ups$ and $\dualups$ denote the stationary distributions of the forward and dual sequences $(\xi_{n}, R_{n})_{n \ge 0}$ and $(\dualxi_{n}, \dualR_{n})_{n \ge 0}$, respectively. Then,
\begin{equation}
	\ups\,=\,\cL\left(\xi_0, \sum_{n \ge 1} \dualPi_{n-1} \dualB_{n}\right)\quad\text{and}\quad \dualups\,=\,\cL\left(\xi_0, \sum_{n \ge 1} \itPi_{n-1} B_{n}\right),
\end{equation}
where $\cL(X)$ denotes the law of the random variable $X$.

\subsection{Induced Markov random walk and its dual}

Define $S_0 := 0$ and $S_{n} := \log|\itPi_{n}|$ for $n \ge 1$. Then the sequence $(\xi_{n}, S_{n})_{n \ge 0}$ forms a Markov random walk (\MRW), as the increments $\zeta_{n} := S_{n} - S_{n-1} = \log|A_{n}|$ for $n \ge 1$ are modulated by $\Xi$: they are conditionally independent given $\Xi$ and have temporally homogeneous conditional distributions,
\begin{align*}
	\Prob(\zeta_{n} \in D\,|\,\Xi)\ =\ F_{\xi_{n-1}, \xi_{n}}(\pm \e^{D} \times \R),
\end{align*}
for measurable $D \subset \R$ and all $n\in\N$, where $\pm \e^D := \{ \pm \e^x : x \in D \}$.
Equivalently, the bivariate process $\Xi^\star := (\xi_{n}, \zeta_{n})_{n \ge 0}$, with $\zeta_0 := 0$ (for definiteness), defines a Markov chain on $\cS \times \R$ with transition kernel
\begin{equation*}
	\Gamma(i, \{j\} \times D)\,:=\, \Prob_{i}[\xi_{1} = j, \log|A_{1}| \in D]\,=\, p_{ij} F_{ij}(\pm \e^D \times \R).
\end{equation*}
Hence, the law of $(\xi_{n}, \zeta_{n})$ conditional on the past depends only on $\xi_{n-1}$.
Moreover, $\Xi^\star$ has stationary distribution $\pi^\star$ given by
\begin{equation*}
	\pi^\star(\{j\} \times D)\ =\ \Prob_\pi [\xi_{1}=j,\log|A_{1}|\in D]\ =\ \sum_{i \in \cS} \pi_{i} p_{ij} F_{ij}(\pm \e^D \times \R).
\end{equation*}

The \MRW\ $(\xi_{n}, S_{n})_{n \ge 0}$ plays a central role in our analysis. Its renewal-theoretic properties -- after a suitable change of measure $\Prob \rightsquigarrow \wh{\Prob}$ -- are closely linked to the tail behavior of the stationary distribution of $(R_{n})_{n \ge 0}$, in line with Goldie's implicit renewal approach. Our extension addresses the case where the random walk is modulated by a finite Markov chain; see also \cite{AlsBroBur:23}. The measure change is detailed in Subsection~\ref{subsec:measure change}.

\vspace{.2cm}
Using \eqref{eq:def Pi(n)}, we can extend $(\xi_{n}, S_{n})_{n \ge 0}$ to a doubly infinite sequence by setting
\begin{equation*}
	S_{n}\,:=\, \log|\itPi_{n}| \quad \text{for } n = -1, -2, \ldots,
\end{equation*}
and then define its time-reversed (dual) process as
$$ (\wh{\xi}_{n}, \wh{S}_{n})_{n \in \Z}\,:=\,(\xi_{-n}, S_{-n})_{n \in \Z}. $$
In particular, $\wh{S}_0 = 0$ and
$$ \wh{S}_{n}\,=\,\log|\itPi_{-n}|\,=\,\log|\wh{\itPi}_{n}|\,=\,\sum_{k=1}^n \wh{\zeta}_k\quad\text{for }n\in\N,
$$
where $\wh{\zeta}_{n} := \log|\wh{A}_{n-1}| = \log|A_{-n+1}|$. The associated Markov chain $(\wh{\xi}_{n}, \wh{\zeta}_{n})_{n \ge 0}$, with $\wh{\zeta}_0 := 0$, has transition kernel
\begin{equation*}
	\wh{\Gamma}(i, \{j\} \times D)\,=\,\wh{p}_{ij} \wh{F}_{ij}(\pm \e^D \times \R)\,=\,\frac{\pi_j p_{ji}}{\pi_{i}} F_{ji}(\pm \e^D \times \R),
\end{equation*}
for $i,j \in \cS$ and measurable $D \subset \R$.

\subsection{The Cram\'er transforms of $P$ and its dual $\dualP$}\label{subsec:Cramer transform}

The Cram\'er transforms of the transition matrix $P$ and its dual $\dualP$, along with their stochastic normalizations, play a central role in our analysis. In this subsection, we collect the relevant properties, focusing primarily on $P$, with comments on the necessary (but minor) modifications for $\dualP$ provided at the end.

\vspace{.1cm}
The \emph{Cram\'er transform} of $P$ is defined as the matrix $\sfP(\theta) = (\sfp_{ij}(\theta))_{i,j\in\cS}$, where
\begin{equation}\label{eq:Cramer transform of P}
	\sfp_{ij}(\theta)\ :=\ \Erw_{i}\Big[|A_{1}|^{\theta}\1_{\{\xi_{1}=j\}} \Big] \ =\ \Erw_{i}\Big[\e^{\theta S_{1}}\1_{\{\xi_{1}=j\}} \Big]\ =\ p_{ij}\Phi_{ij}(\theta),
\end{equation}
for all $\theta \ge 0$, and where $\Phi_{ij}(\theta) := \Erw_{i}[\e^{\theta S_{1}}\,|\,\xi_{1} = j]$. 

\vspace{.1cm}
The entries of $\sfP(\theta)$ are finite if and only if $\Erw_{i}[\e^{\theta S_{1}}] < \infty$ for all $i\in\cS$, which (since $\cS$ is finite) is equivalent to $\Erw_{\pi}[\e^{\theta S_{1}}] < \infty$. The set of such $\theta\ge 0$ defines the canonical domain $\bbD$ of $\sfP(\theta)$. In fact, $\sfP(\cdot)$ admits a holomorphic extension to the complex strip
$$
\bbD_{\C} := \{ z \in \C : \Re(z) \in \bbD\},
$$
where $\Re(z)$ denotes the real part of $z$. If Assumption \textsf{(B1)} holds, then $[0, \kappa] \subset \bbD$.
Note that $\sfP(0) = P$, and each entry $\sfp_{ij}(\theta)$ is convex on $\bbD$, with derivative
\begin{equation}\label{eq:derivative Cramer transform of P}
	\sfp_{ij}'(\theta)\ :=\ \Erw_{i}\left[\e^{\theta S_{1}} S_{1} \1_{\{\xi_{1} = j\}} \right] \ =\ p_{ij} \Phi_{ij}'(\theta), \quad i,j \in \cS.
\end{equation}
Furthermore, for any $i,j\in\cS$,
\begin{equation}\label{eq:p(i,j) positive}
	\sfp_{ij}(\theta)\,>\,0\text{ for some }\theta\in\bbD\quad\Longrightarrow\quad\sfp_{ij}(\theta)\,>\,0\text{ for all }\theta\in\bbD,
\end{equation}
and the same holds true for the entries $\sfp_{ij}^{n}(\theta)$ for any $n\ge 2$.

\vspace{.1cm}
Let $\rho(\theta)$ denote the dominant eigenvalue of $\sfP(\theta)$, and let $u(\theta) = (u_{i}(\theta))_{i\in\cS}$ and $v(\theta) = (v_{i}(\theta))_{i\in\cS}$ be corresponding positive left and right eigenvectors, respectively, normalized such that
$$ \sum_{i\in\cS} u_{i}(\theta)\ =\ \sum_{i\in\cS} u_{i}(\theta)v_{i}(\theta) = 1. $$
Under \textsf{(B1)} and by Remark~\ref{rem:Pkappa irreducible}, the matrix $\sfP(\kappa)$ is irreducible with $\rho(\kappa) = 1$. Hence, by \eqref{eq:p(i,j) positive}, $\sfP(\theta)$ for \emph{any} $\theta \in\bbD$ is irreducible as well, and $\rho(\theta) > 0$. It follows that the matrix $\rho(\theta)^{-1} \sfP(\theta)$ is irreducible and nonnegative, with maximal eigenvalue 1 and the same eigenvectors (up to scaling) as $\sfP(\theta)$. As shown in \cite[p.~360]{Alsmeyer:14}, such a matrix -- termed \emph{quasi\-stochastic} there -- can be transformed into a proper irreducible stochastic matrix -- referred to as the \emph{normalized Cram\'er transform} -- via
\begin{equation}\label{eq:Phat(theta)}
	P(\theta)\ :=\ (p_{ij}(\theta))_{i,j\in\cS}\ :=\ \frac{1}{\rho(\theta)} D(\theta)^{-1} \sfP(\theta) D(\theta) \ =\ \left( \frac{v_j(\theta) \sfp_{ij}(\theta)}{\rho(\theta)v_{i}(\theta)} \right)_{i,j\in\cS},
\end{equation}
where $D(\theta) := \diag[v_{i}(\theta) : i \in \cS]$. Note that $p_{ij}(\theta) > 0$ if and only if $\sfp_{ij}(\theta) > 0$.

\vspace{.1cm}
The unique stationary distribution of $P(\theta)$ -- which is also its unique normalized positive left eigenvector -- is given by
\begin{equation}\label{eq:pihat(theta)}
	\pi(\theta)\ =\ D(\theta) u(\theta)\ =\ \big( u_{i}(\theta) v_{i}(\theta) \big)_{i \in \cS}.
\end{equation}

Let $\sfp_{ij}^{n}(\theta)$ and $p_{ij}^{n}(\theta)$ denote the $(i,j)$-entries of $\sfP(\theta)^n$ and $P(\theta)^n$, respectively. Then, using \eqref{eq:Phat(theta)}, we obtain the relation
\begin{equation}\label{eq:Phat^n(theta)}
	p_{ij}^{n}(\theta)\ =\ \frac{v_j(\theta)\, \sfp_{ij}^{n}(\theta)}{\rho(\theta)^n\, v_{i}(\theta)},
\end{equation}
for all $n\in\N$ and $i,j \in \cS$.

\smallskip
The following lemma extends Lemmata~4.1 and 4.2 from \cite{AlsBroBur:23} to the more general setting considered here, where the modulating chain is not restricted to have only two states.

\begin{Lemma} \label{lem:Ptheta vs Lambda}
	For any $\theta \in \bbD$ and $i,j \in \cS$, we have
	\begin{gather}
		\sfp_{ij}^{n}(\theta) \,=\, \Erw_{i}\big[\itPi_{n}^{\theta} \1_{\{\xi_{n}=j\}} \big] \,=\, \Erw_{i}\big[\e^{\theta S_{n}} \1_{\{\xi_{n}=j\}} \big], \label{eq: Ptheta vs Pi n} \\
		\lim_{n \to \infty} \frac{1}{n} \log \sum_{j \in \cS} \sfp_{ij}^{n}(\theta) \ =\ \lim_{n \to \infty} \frac{1}{n} \log \Erw_{i}[\e^{\theta S_{n}}] \ =\ \log \rho(\theta). \label{eq:Lip vs rho}
	\end{gather}
	Moreover, the function $\theta \mapsto \log \rho(\theta)$ is continuous and convex on its canonical domain $\bbD$, and smooth (i.e., infinitely differentiable) on the interior of $\bbD$.  As a direct consequence, if $\rho(\kappa)=1$ for some $\kappa>0$, then $\rho(\theta)<1$ for $0<\theta<\kappa$.
\end{Lemma}

\smallskip

To prove Lemma~\ref{lem:Ptheta vs Lambda}, we will use the following proposition concerning the analyticity of the dominant eigenvalue $\rho(\theta)$ of an analytic family of nonnegative matrices $\sfP(\theta)$, extended holomorphically to a complex domain. The result follows from the implicit function theorem for analytic maps, see e.g.\ \cite[Thm.\ 8.6]{KaupKaup:83}.

\begin{Prop}\label{prop:PF eigenvalue smooth}
	Let $\bbD = \{ z \in \C : 0 < \Re(z) < \theta_0 \}$ for some $\theta_0 > 0$, and let $\{\sfP(z) : z \in \bbD\}$ be a holomorphic family of matrices such that $\sfP(\theta)$ is nonnegative with a positive (and hence simple) maximal eigenvalue $\rho(\theta)$ for each real $\theta \in \bbD$. Then the function $\theta \mapsto \rho(\theta)$ is infinitely differentiable on $(0, \theta_0)$.
\end{Prop}

\begin{proof}[Proof of Lemma~\ref{lem:Ptheta vs Lambda}]
	The identity in \eqref{eq: Ptheta vs Pi n} can be established by induction on $n$. For the base case $n = 1$, it follows directly from the definition of $\sfp_{ij}(\theta)$ in \eqref{eq:Cramer transform of P}. For the inductive step, assume the identity holds for some $n \ge 1$. Then
	\begin{align*}
		\sfp_{ij}^{n+1}(\theta)\ &=\ \sum_{k \in \cS} \sfp_{ik}(\theta) \sfp_{kj}^{n}(\theta)\ 
		=\ \sum_{k \in \cS} \Erw_{i}\left[ \e^{\theta S_{1}} \1_{\{\xi_{1} = k\}} \right] \Erw_{k}\left[ \e^{\theta S_{n}} \1_{\{\xi_{n} = j\}} \right] \\
		&=\ \sum_{k \in \cS} \Erw_{i}\left[ \e^{\theta S_{1}} \1_{\{\xi_{1} = k\}} \cdot \e^{\theta (S_{n+1} - S_{1})} \1_{\{\xi_{n+1} = j\}} \right] \\
		&=\ \Erw_{i}\left[ \e^{\theta S_{n+1}} \1_{\{\xi_{n+1} = j\}} \right],
	\end{align*}
	which completes the induction.
	
	\smallskip
	To prove \eqref{eq:Lip vs rho}, we assume for simplicity that the modulating chain $(\xi_{n})_{n \ge 0}$, and hence $P$, is aperiodic. The extension to the periodic case via decomposition into cyclic classes is standard and omitted.
	
	Under aperiodicity, the normalized Cram\'er transforms $P(\theta)$ are themselves aperiodic for all $\theta \in \bbD$. Thus, by the ergodic theorem for finite Markov chains, we have
	$$ \lim_{n \to \infty} p_{ij}^{n}(\theta)\ =\ \pi_{j}(\theta)\ =\ u_{j}(\theta) v_{j}(\theta)\ >\ 0 $$
	for all $i,j \in \cS$. Applying this to the representation in \eqref{eq:Phat^n(theta)} gives
	\begin{align*}
		\frac{1}{n} \log \sum_{j \in \cS} \sfp_{ij}^{n}(\theta)\ 
		&=\ \log \rho(\theta) + \frac{1}{n} \log \sum_{j \in \cS} \frac{p_{ij}^{n}(\theta) v_{i}(\theta)}{v_j(\theta)} \\
		&=\ \log \rho(\theta) + \frac{1}{n} \log \sum_{j \in \cS} \left( u_j(\theta) v_{i}(\theta) + o(1) \right) \\
		&=\ \log \rho(\theta) + o(1),
	\end{align*}
	as $n\to\infty$, which proves \eqref{eq:Lip vs rho}.
	
	\smallskip
	Convexity of the function $\theta \mapsto \log \rho(\theta)$ follows directly from \eqref{eq:Lip vs rho} and standard properties of logarithmic moment generating functions (see also \cite[Cor.~2]{Saporta:05}). Smoothness of $\rho(\theta)$ on the interior of $\bbD$ is a classical result, and follows from Proposition~\ref{prop:PF eigenvalue smooth}.
\end{proof}

Turning to the dual transition matrix $\dualP$, it can easily be verified that the associated Cram\'er transform $\dualsfP(\theta)$ is related to $\sfP(\theta)$ by
$$
\dualsfP(\theta)\ =\ \itPi(\theta)^{-1}\sfP(\theta)^{\top}\itPi(\theta)\ =\ \bigg(\frac{\pi_{j}(\theta)\sfp_{ji}(\theta)}{\pi_{i}(\theta)}\bigg)_{i,j\in\cS},
$$
where $\itPi(\theta):=\diag[\pi_{i}(\theta):i\in\cS]$. The matrix $\dualsfP(\theta)$ has the same dominant eigenvalue as $\sfP(\theta)$, thus $\rho(\theta)$, but its left and right eigenvectors are different. Specifically,
\begin{gather}
	\dualu(\theta)\ =\ \itPi(\theta)v(\theta)\ =\ (\pi_{i}(\theta)v_{i}(\theta))_{i\in\cS}\ =\ (u_{i}(\theta)v_{i}(\theta)^2)_{i\in\cS}, \label{eq: def uhat}
	\shortintertext{and}
	\dualv(\theta)\ =\ u(\theta)^{\top}\itPi(\theta)^{-1}\ =\ (v_{i}(\theta)^{-1})_{i\in\cS}, \label{eq:def vhat}
\end{gather}
under the normalization $\sum_{i\in\cS}\dualu_{i}(\theta) = \dualu(\theta)^{\top}\dualv(\theta) = 1$.

\smallskip

Analogous to \eqref{eq:Phat(theta)}, the stochastic normalization of the quasistochastic matrix $\rho(\theta)^{-1}\dualsfP(\theta)$ is given by
$$
\dualP(\theta)\ :=\ \frac{1}{\rho(\theta)}\dualD(\theta)^{-1}\dualsfP(\theta)\dualD(\theta)\ =\ \frac{1}{\rho(\theta)}\dualD(\theta)^{-1}\,\itPi(\theta)^{-1}\sfP(\theta)^{\top}\itPi(\theta)\,\dualD(\theta),
$$
where
$$
\dualD(\theta)\ :=\ \diag[\pi_{i}(\theta)^{-1}u_{i}(\theta):i\in\cS],
$$
so that
$$
\itPi(\theta)\dualD(\theta)\ =\ \diag[u_{i}(\theta):i\in\cS].
$$

\smallskip

Finally, since $\dualu_{i}(\theta)\dualv_{i}(\theta) = v_{i}(\theta)u_{i}(\theta) = \pi_{i}(\theta)$ for all $i\in\cS$, it follows that $\dualP(\theta)$ and $P(\theta)$ have the same stationary distribution $\pi(\theta)$.

\subsection{Measure change}\label{subsec:measure change}

Let $\FF = (\cF_{n})_{n \ge 0}$ be the natural filtration of $(\xi_{n}, A_{n}, B_{n})_{n \ge 0}$ and put $\cF_\infty := \sigma\{\cF_{n} : n \ge 0\}$.
Note that $(\xi_{n}, \zeta_{n})_{n \ge 0}$ and the \MRW\ $(\xi_{n}, S_{n})_{n \ge 0}$ are $\cF$-adapted.

Assuming \textsf{(B1)} -- and thus in particular $\rho(\kappa) = 1$ for some positive $\kappa \in \bbD$ -- the convexity of $\rho$ implies that $\rho(\theta) < 1$ for all $\theta \in (0, \kappa)$, $\rho'(0) < 0$, and $\rho'(\kappa) > 0$.  Recall from Remark~\ref{rem:Pkappa irreducible} that $\sfP(\kappa)$ is irreducible, and by \eqref{eq:p(i,j) positive} the same holds for each $\sfP(\theta)$, $\theta\in\bbD$. Hence, the stochastic normalization (see \eqref{eq:Phat(theta)})
$$ P(\theta)\ =\ D(\theta)^{-1}\sfP(\theta)D(\theta)\ =\ \Bigg(\frac{\sfp_{ij}(\theta)v_{j}(\theta)}{v_{i}(\theta)}\Bigg)_{i,j\in\cS} $$
is an irreducible transition matrix with unique stationary law $\pi(\theta)=(u_{i}(\theta)v_{i}(\theta))_{i\in\cS}$ (see \eqref{eq:pihat(theta)}) for all \emph{positive} $\theta \in \bbD$.

\vspace{.1cm}
For any such $\theta$ and $i \in \cS$, we have that
$$
\left(\e^{\theta S_{n}} v_{\xi_{n}}(\theta) \rho(\theta)^{-n} \right)_{n \ge 0}
$$
is a positive $\cF$-martingale under any $\Prob_s$, $s \in\cS$. Indeed, for any $i\in\cS$, on the event $\{\xi_{n} = i\} \in \cF_{n}$,
\begin{align*}
	\Erw\bigg[\frac{e^{\theta S_{n+1}}v_{\xi_{n+1}}(\theta)}{\rho(\theta)^{n+1}}\bigg|\cF_{n}\bigg]
	&= \frac{e^{\theta S_{n}}}{\rho(\theta)^{n+1}} \Erw_{i} \left[\e^{\theta S_{1}} v_{\xi_{1}}(\theta) \right] \\
	&= \frac{e^{\theta S_{n}}}{\rho(\theta)^{n+1}} \sum_{j \in \cS} \sfp_{ij}(\theta) v_j(\theta) \\
	&= \frac{e^{\theta S_{n}} v_{i}(\theta)}{\rho(\theta)^n} \sum_{j \in \cS} p_{ij}(\theta) \\
	&= \frac{e^{\theta S_{n}} v_{i}(\theta)}{\rho(\theta)^n},
\end{align*}
where Lemma~\ref{lem:Ptheta vs Lambda} has been used in the second step, and \eqref{eq:Phat^n(theta)} in the third.

\vspace{.1cm}
Using this, we now define a new probability measure $\Prob_{i}^{(\theta)}$ on $(\Omega, \cF_\infty)$ as follows: for any $n \in \N$ and bounded measurable function 
$f : \cS \times (\cS \times \IRg \times \R)^n \to \R$, set
\begin{equation}\label{eq:def Ptheta_{i}}
	\begin{split}
		\Erw_{i}^{(\theta)}&\big[f(\xi_{0},(\xi_{1},A_{1},B_{1}),\ldots,(\xi_{n},A_{n},B_{n}))\big] \\
		&:=\ \frac{1}{v_{i}(\theta)\rho(\theta)^{n}}\,\Erw_{i}\big[e^{\theta S_{n}}v_{\xi_{n}}(\theta)f(\xi_{0},
		(\xi_{1},A_{1},B_{1}),\ldots,(\xi_{n},A_{n},B_{n}))\big].
	\end{split}
\end{equation}

As stated in the next lemma, the process $(\xi_{n}, \zeta_{n})_{n \ge 0}$ remains a Markov chain on $\cS \times \R$ under the changed measure $\Prob_{i}^{(\theta)}$, but now with a modified transition kernel $\Gamma^{(\theta)}$ given by
\begin{equation}\label{eq:Qtheta g}
	\Gamma^{(\theta)} g(i, x)\,:=\,\frac{1}{v_{i}(\theta)\rho(\theta)}\,\Erw_{i}\Big[e^{\theta \zeta_{1}} v_{\xi_{1}}(\theta) g(\xi_{1}, \zeta_{1})\Big]
\end{equation}
for all bounded functions $g: \cS \times \R \to \R$. Hence, the conditional law of $(\xi_{1}, \zeta_{1})$ given the initial state $(i, x)$ depends only on $i$ and not on $x$.

\begin{Lemma}\label{lem:Pdelta^kappa case 1}
	For any $\theta\in\bbD$, the following assertions hold under the new probability measures $\Prob_{i}^{(\theta)}$, $i\in\cS$:
	\begin{itemize}\itemsep2pt
		\item[(a)] The sequence $(\xi_{n})_{n\ge 0}$ is an irreducible Markov chain on $\cS$ with transition matrix $P(\theta)$ as defined in \eqref{eq:Phat(theta)} and unique stationary distribution $\pi(\theta)$ given by \eqref{eq:pihat(theta)}.
		\item[(b)] The sequence $(\xi_{n},\zeta_{n})_{n\ge 0}$ is a Markov chain with transition operator $Q^{(\theta)}$. Its first component has initial state $i$ under $\Prob_{i}^{(\theta)}$ and is stationary under $\Prob^{(\theta)}:=\sum_{i}\pi_{i}(\theta)\Prob_{i}^{(\theta)}$.
		\item[(c)] The sequence $(\xi_{n},A_{n},B_{n})_{n\ge 1}$ is stationary under $\Prob^{(\theta)}$.
		\item[(d)] The sequence $(\xi_{n},S_{n})_{n\ge 0}$ is a \MRW\ with driving chain $(\xi_{n})_{n\ge 0}$ and $S_{0}=0$.
		\item[(e)] The stationary drift of $(\xi_{n},S_{n})_{n\ge 0}$ under $\Prob^{(\theta)}$ equals
	\end{itemize}
	\begin{equation}\label{eq:stationary drift}
		\Erw^{(\theta)}S_{1}\ =\ \frac{\rho'(\theta)}{\rho(\theta)}\ =\ \sum_{i,j\in\cS}\frac{u_{i}(\theta)v_{j}(\theta)}{\rho(\theta)}\,\Erw_{i}\Big[|A_{1}|^{\theta}\log|A_{1}|\1_{\{\xi_{1}=j\}}\Big]\ =\ u(\theta)^{\top}P'(\theta)v(\theta).
	\end{equation}
	\begin{itemize}\itemsep2pt
		\item[] This drift is finite for all $\theta$ in the interior of $\bbD$, and negative at $\theta = 0$, i.e., under $\Prob^{(0)}=\Prob_{\pi}$. In fact $\Erw S_{1}=\Erw\log|A_{1}|=\rho'(0)$.
		\item[(f)] If $\kappa > 0$ exists such that $\rho(\kappa) = 1$, then the stationary drift under $\Prob^{(\kappa)}$ is positive, but it can be infinite if $\kappa$ is a boundary point of $\bbD$. In the latter case, \eqref{eq:stationary drift} still holds with $\rho'(\kappa):=\lim_{\theta\uparrow\kappa}\rho'(\theta)$.
	\end{itemize}
\end{Lemma}

\begin{proof}
	Except for part~(c), the proof of this lemma is omitted, as it follows with only minor modifications from the proofs of Lemmata 4.3 and 4.4 in \cite{AlsBroBur:23}, which treat the case $|\cS| = 2$.
	To derive the second equality in \eqref{eq:stationary drift}, one simply uses the definition of $\Prob^{(\theta)}$, namely,
	\begin{equation}\label{eq:def Ptheta}
		\begin{split}
			\Erw_\pi^{(\theta)} &f(\xi_{0},(\xi_{1},A_{1},B_{1}),\ldots,(\xi_{n},A_{n},B_{n}))\\
			&= \sum_{i \in \cS} \frac{u_{i}(\theta)}{\rho(\theta)^n}\, \Erw_{i}\e^{\theta S_{n}} v_{\xi_{n}}(\theta) f(\xi_0, (\xi_{1}, A_{1}, B_{1}),\ldots,(\xi_{n}, A_{n}, B_{n})),
		\end{split}
	\end{equation}
	for any bounded measurable function $f$. Combined with the definition of $\pi(\theta)$, this yields the expression for $\Erw^{(\theta)} S_{1}$.
	
	\vspace{.2cm}
	We now turn to part~(c). Define, for $n \in \N$, $i_0, \ldots, i_{n} \in \cS$, $x_{1}, \ldots, x_{n} \in \R$, and any bounded measurable function $f : \cS \times (\cS \times \R)^n \to \R$,
	\begin{align*}
		\Gamma f(i_0, (i_k, x_k)_{k=1,\ldots,n})
		\,:=\,\Erw_{i_0}\big[f(\xi_0, (\xi_k, A_k, B_k)_{k=1,\ldots,n}) \,\big|\, (\xi_k, \zeta_k) = (i_k, x_k),\, k = 1, \ldots, n\big].
	\end{align*}
	Then, using the stationarity of $(\xi_{n}, \zeta_{n})_{n \ge 0}$ under $\Prob^{(\theta)}$ from part~(b), and recalling that $\pi_{i}(\theta) = u_{i}(\theta) v_{i}(\theta)$ for $i \in \cS$, we obtain:
	\begin{align*}
		\Erw_\pi^{(\theta)} &f(\xi_m, (\xi_{m+k}, A_{m+k}, B_{m+k})_{k=1,\ldots,n})\\
		&= \sum_{i \in \cS} \frac{u_{i}(\theta)}{\rho(\theta)^{m+n}} \Erw_{i}\Big[\e^{\theta S_{m+n}} v_{\xi_{m+n}}(\theta)\,\Gamma f(\xi_m, (\xi_{m+k}, \zeta_{m+k})_{k=1,\ldots,n})\Big]\\
		&= \Erw_\pi^{(\theta)} \Gamma f(\xi_m, (\xi_{m+k}, \zeta_{m+k})_{k=1,\ldots,n}) \\
		&= \Erw_\pi^{(\theta)} \Gamma f(\xi_0, (\xi_k, \zeta_k)_{k=1,\ldots,n}) \\
		&= \Erw_\pi^{(\theta)} f(\xi_0, (\xi_k, A_k, B_k)_{k=1,\ldots,n}),
	\end{align*}
	where the third equality follows by stationarity of $(\xi_{n}, \zeta_{n})_{n \ge 0}$. This shows that the law of $(\xi_{n}, A_{n}, B_{n})_{n \ge 1}$ under $\Prob^{(\theta)}$ is invariant under time shifts, i.e., the sequence is stationary.
\end{proof}

Choosing the dual filtration as $\dualcF_{0}:=\cF_{0}$,
\begin{align*}
	\dualcF_{n}\,:=\,\sigma\big\{\dualxi_{0},(\dualxi_{k},\dualA_{k},\dualB_{k}):1\le k\le n\big\}\,=\,\sigma\big\{\xi_{0},(\xi_{-k},A_{-k+1}, B_{-k+1}):1\le k\le n\big\}
\end{align*}
for $n\ge 1$, and $\dualcF_{\infty}:=\sigma\{\dualcF_{n},n\ge 0\}$, we define the $\dProb_{i}^{(\theta)}$ on $(\Omega,\dualcF_{\infty})$ by
\begin{equation}\label{eq:def dual Ptheta_{i}}
	\begin{split}
		\dErw_{i}^{(\theta)}&f\big(\dualxi_{0},(\dualxi_{1},\dualA_{1},\dualB_{1}),\ldots,(\dualxi_{n},\dualA_{n},\dualB_{n})\big)\\
		&:=\ \frac{\dErw_{i}e^{\theta\dualS_{n}}\dualv_{\dualxi_{n}}(\theta)f(\dualxi_{0},(\dualxi_{1},\dualA_{1},\dualB_{1}),\ldots,(\dualxi_{n},\dualA_{n},\dualB_{n}))}{\dualv_{i}(\theta)\rho(\theta)^{n}}\\
		&=\ \frac{\Erw_{i}\e^{\theta S_{-n}}\dualv_{\xi_{-n}}(\theta)f(\xi_{0},(\xi_{-1},A_{0},B_{0}),\ldots,(\xi_{-n},A_{-n+1},B_{-n+1}))}{\dualv_{i}(\theta)\rho(\theta)^{n}}\\
		&=\ \frac{v_{i}(\theta)}{\rho(\theta)^{n}}\,\Erw_{i}\Bigg[\frac{e^{\theta S_{-n}}f(\xi_{0},(\xi_{-1},A_{0},B_{0}),\ldots,(\xi_{-n},A_{-n+1},B_{-n+1}))}{v_{\xi_{-n}}(\theta)}\Bigg].
	\end{split}
\end{equation}
for all $n\in\N$ and bounded $f$ as in \eqref{eq:def Ptheta_{i}}, and for $\theta\in\bbD$ and $i\in\cS$
(note that \eqref{eq:def vhat} has been utilized to get the last line). These probability measures form the exact dual analogs of the $\Prob_{i}^{(\theta)}$, and a dual analog of the previous lemma for the \MRW\ $(\dualxi_{n},\dualS_{n})_{n\ge 0}$ and the associated sequence $(\dualxi_{n},\dualzeta_{n})_{n\ge 0}$ under the $\dProb_{i}^{(\theta)}$ can now be stated in a straightforward manner. We refrain from doing so, but point out that the stationary law of $(\dualxi_{n})_{n\ge 0}$ under the $\dProb_{i}^{(\theta)}$ is also $\pi(\theta)$ and that (as a consequence) $(\dualxi_{n},\dualzeta_{n})_{n\ge 0}$ is a stationary Markov chain under $\dProb^{(\theta)}:=\sum_{i\in\cS}\pi_{i}(\theta)\dProb_{i}^{(\theta)}$, where
\begin{equation}\label{eq:def dual Ptheta}
	\begin{split}
		&\dErw^{(\theta)}f\big(\dualxi_{0},(\dualxi_{1},\dualA_{1},\dualB_{1}),\ldots,(\dualxi_{n},\dualA_{n},\dualB_{n})\big)\\
		&=\ \dErw^{(\theta)}f\big(\xi_{0},(\xi_{-1},A_{0},B_{0}),\ldots,(\xi_{-n},A_{-n+1},B_{-n+1})\big)\\
		&=\ \sum_{i\in\cS}\frac{\pi_{i}(\theta)}{\dualv_{i}(\theta)\rho(\theta)^{n}}\,\Erw_{i}e^{\theta S_{-n}}\dualv_{\xi_{-n}}(\theta)f(\xi_{0},(\xi_{-1},A_{0},B_{0}),\ldots,(\xi_{-n},A_{-n+1},B_{-n+1}))\\
		&=\ \sum_{i\in\cS}\frac{v_{i}(\theta)\pi_{i}(\theta)}{\rho(\theta)^{n}}\,\Erw_{i}\Bigg[\frac{e^{\theta S_{-n}}f(\xi_{0},(\xi_{-1},A_{0},B_{0}),\ldots,(\xi_{-n},A_{-n+1},B_{-n+1}))}{v_{\xi_{-n}}(\theta)}\Bigg].
	\end{split}
\end{equation}
for bounded functions $f$. Moreover, the sequence $(\dualxi_{n},\dualA_{n-1},\dualB_{n-1})_{n\ge 1}$ is stationary under $\dProb^{(\theta)}$ and $\dErw^{(\theta)}\dualS_{1}=\Erw^{(\theta)}S_{1}$ for any $\theta\in\bbD$.

\subsection{Lattice-type of $(\xi_{n},S_{n})_{n\ge 0}$ and its dual}\label{subsec:lattice-type}

Let $\kappa$ be given by \textsf{(B1)}. By definition, the \MRW\ $(\xi_{n},S_{n})_{n\ge 0}$ is nonlattice under $\Prob^{(\kappa)}$ if
\begin{equation*}
	\Prob^{(\kappa)}\big[\zeta_{1}-a_{\xi_{1}}+a_{\xi_{0}}\in\, d\Z\big]\,<\,1
\end{equation*} 
for any $d>0$ and any $\{a_{i}:i\in\cS\}\subset [0,d)$, see e.g.~\cite[p.\,106]{Alsmeyer:97}. As one can easily see, this is equivalent to \textsf{(B3)}, and by the same reasoning, \textsf{(B3)} ensures that the dual \MRW\ $(\dualxi_{n},\dualS_{n})_{n\ge 0}$ is nonlattice under $\dProb^{(\kappa)}$.\\

\section{Condition \textsf{(B1)}: Making life easier by geometric sampling}

Given assumptions \textsf{(B1--4)} of Theorem~\ref{thm:main_{1}}, we now show that in \textsf{(B1)} the irreducible transition matrix $P$ of the driving chain, as well as its Cram\'er transform $\sfP(\kappa)$, may be assumed to be \emph{positive} without loss of generality. That is, we may assume that
$$
p_{ij} > 0 \quad \text{and} \quad 
\Prob_{i}[A_{1} \ne 0 \mid \xi_{1} = j] = F_{ij}(\R\backslash\{0\}) > 0 
\quad \text{for all } i, j \in \cS.
$$
Since $\sfP(\theta)$ is positive if and only if its normalization $P(\theta)$ is positive, we focus on the latter, as its entries are transition probabilities and thus admit a probabilistic interpretation.

Furthermore, we show that if 
$$
\Prob[A_{1} > 0] \wedge \Prob[A_{1} < 0] > 0,
$$
then the transition probabilities
\begin{equation}\label{eq:def pplusminus}
	p_{ij}^{-}(\kappa) := \Prob_{i}^{(\kappa)}[\xi_{1} = j, A_{1} < 0] 
	\quad \text{and} \quad 
	p_{ij}^{+}(\kappa) := \Prob_{i}^{(\kappa)}[\xi_{1} = j, A_{1} > 0]
\end{equation}
can also be assumed to be positive. All these conclusions extend to the dual quantities $\dualP$, $\dualP(\kappa)$, and $\dualp_{ij}^{\,\pm}(\kappa)$ as well. The key tool is \emph{geometric sampling}.

\vspace{.1cm}
Let $(\vth_{n})_{n \in \Z}$ be a doubly infinite sequence of random times, independent of $(\xi_{n}, A_{n}, B_{n})_{n \in \Z}$, with $\vth_0 = 0$ and i.i.d.~increments $\vth_{n} - \vth_{n-1}$ having a geometric distribution:
$$
\Prob[\vth_{n} - \vth_{n-1} = k] = \frac{1}{2^k} 
\quad \text{for all } k \in \N,\; n \in \Z.
$$
Given the stationary \MMLIFS\ $(\xi_{n}, R_{n})_{n \ge 0}$, we define the geometrically sampled process 
$$ (\xi_{n}^{*}, R_{n}^{*}) := (\xi_{\vth_{n}}, R_{\vth_{n}})\quad n\ge 0, $$
and note, without proof, that this sequence is again a stationary \MMLIFS\ with the same marginal law and associated sequence $(A_{n}^{*}, B_{n}^{*})_{n \in \Z}$ defined by
$$
(A_{n}^{*}, B_{n}^{*}) := 
\left( A_{\vth_{n-1}+1} \cdots A_{\vth_{n}},\; 
\Psi_{\vth_{n}} \circ \cdots \circ \Psi_{\vth_{n-1}+1}(0) \right),
$$
modulated by $(\xi_{n}^{*})_{n \in \Z}$.

\begin{Lemma}\label{lem_geosampling}
	Let $P^{*} = (p_{ij}^{*})_{i,j \in \cS}$ be the transition matrix of $(\xi_{n}^{*})_{n \ge 0}$ under $\Prob = \Prob^{(0)}$, and let $P^{*}(\theta) = (p_{ij}^{*}(\theta))_{i,j \in \cS}$ denote its normalized Cram\'er transform. Define also
	$$
	p_{ij}^{*\pm}(\theta) := \Prob_{i}^{(\theta)}[\xi_{1}^{*} = j,\; A_{1}^{*} \gtrless 0]
	\quad \text{for } \theta \in [0, \kappa],
	$$
	and let $\dualP^{*}$, $\dualP^{*}(\theta)$, and $\dualp_{ij}^{*\pm}(\theta)$ denote the corresponding quantities for the geometrically sampled dual chain $\dualxi_{n}^{*} := \xi_{-\vth_{-n}}$, $n \ge 0$. Then the following holds:
	\begin{itemize}[leftmargin=1cm]\itemsep3pt
		\item[(a)] The matrix
		$$ P^{*}\ =\ \sum_{n \ge 1} \frac{1}{2^n} P^n $$
		is positive, i.e., $p_{ij}^{*} > 0$ for all $i, j \in \cS$, and the same holds for the matrix
		$$ P^{*}(\theta)\ =\ \sum_{n \ge 1} \frac{1}{2^n} P(\theta)^n\quad\text{for each }\theta\in [0, \kappa]. $$
		\item[(b)] If $\Prob[A_{1}>0\hspace{.8pt}]\wedge\Prob[A_{1}<0\hspace{.8pt}]>0$, then $p_{ij}^{*\pm}(\theta)>0$ for all $i,j\in\cS$ and $\theta\in (0,\kappa]$.
		\item[(c)] The same conclusions hold for the dual quantities $\dualP^{*}$, $\dualP^{*}(\theta)$, and $\dualp_{ij}^{*\pm}(\theta)$ for each $\theta\in (0,\kappa]$.
	\end{itemize}
\end{Lemma}

Note that geometric sampling of the dual chain $(\xi_{n}^{*})_{n \ge 0}$ yields the same result as taking the dual of the geometrically sampled chain $(\xi_{n})_{n \ge 0}$. In other words, duality and geometric sampling commute. As a consequence, we have
$$
\dualP^{*}\ =\ \itPi^{-1} P^{*} \itPi\ =\ \left( \frac{\pi_j p_{ji}^{*}}{\pi_{i}} \right)_{i,j \in \cS},
$$
and an analogous relation holds between $\dualP^{*}(\theta)$ and $P^{*}(\theta)$ for every $\theta \in [0, \kappa]$.

\begin{proof}[Proof of Lemma~\ref{lem_geosampling}]
	(a) Since $P$ is irreducible, $P^m$ is a positive matrix for some $m \in \N$ (recall that $\cS$ is finite). Thus, $P^{*}=\sum_{n\ge 1}2^{-n}P^{n}>0$, as claimed. As explained at the beginning of Subsection \ref{subsec:measure change}, the matrix $\sfP(\theta)$ is irreducible for all $\theta \in [0, \kappa]$, including in particular $\theta = \kappa$. This implies that for each such $\theta$,
	$$ \sfP^{*}(\theta) := \left( \Erw_{i}|A_{1}^{*}|^\theta \,\1_{\{\xi_{1}^{*} = j\}}\right)_{i,j \in \cS}
	\ =\ \sum_{n \ge 1}\frac{1}{2^n}\sfP(\theta)^n $$
	and its normalization $P^{*}(\theta)$ $P^{*}(\theta)$ are positive matrices as well. Moreover, the matrix $\sfP^{*}(\theta)$ has dominant eigenvalue
	$$ \rho^{*}(\theta)\ =\ \frac{\rho(\theta)}{2 - \rho(\theta)}, $$
	with the same left and right eigenvectors as $\sfP(\theta)$. In particular, $\rho^{*}(\kappa) = \rho(\kappa) = 1$.
	
	\vspace{.2cm}
	(b) Let $\theta \in (0, \kappa]$. Since $\theta > 0$, we have
	$$ \Prob^{(\theta)}\left[ \itPi_{n} \ne 0 \text{ for all } n \in \N \right]\,=\,1. $$
	because the transformed measure $\Prob^{(\theta)}$ gives zero weight to paths where $|A_{n}| = 0$ for some $n$.
	Define $\tau_0 := 0$ and let $\tau_{1} < \tau_2 < \cdots$ denote the successive times at which the product $\itPi_{n}$ flips sign, i.e.,
	$$ \tau_{n}\ =\ \inf \left\{ k > \tau_{n-1} : \frac{\itPi_k}{\itPi_{\tau_{n-1}}} < 0 \right\}\ =\ \inf \left\{ k > \tau_{n-1} : A_k < 0 \right\}. $$
	Under the assumption
	$$ \Prob[A_{1} > 0] \wedge \Prob[A_{1} < 0] > 0, $$
	each $\tau_{n}$ is almost surely finite under $\Prob^{(\theta)}$. Moreover, $\itPi_k < 0$ holds whenever $k \in [\tau_{2n-1}, \tau_{2n})$ for some $n \in \N$, hence:
	$$ p_{ij}^{*+}(\theta)\ =\ \Prob_{i}^{(\theta)} \left[ \xi_{\vth_{1}} = j,\; \vth_{1} \in [\tau_{2n}, \tau_{2n+1}) \text{ for some } n \in \N_0 \right], $$
	and similarly,
	$$ p_{ij}^{*-}(\theta)\ =\ \Prob_{i}^{(\theta)}\left[ \xi_{\vth_{1}} = j,\; \vth_{1} \in [\tau_{2n-1}, \tau_{2n}) \text{ for some } n \in \N \right]. $$
	To show that these probabilities are strictly positive, observe that
	\begin{align*}
		p_{ij}^{*+}(\theta)\ &\ge\ \Prob_{i}^{(\theta)}\big[\xi_{1} = j,\,A_{1}\ge 0,\,\vth_{1}=1\big]\ =\ \frac{1}{2}p_{ij}^{+}(\theta)\ >\ 0
		\shortintertext{and} 
		p_{ij}^{*-}(\theta)\ &\ge\ \Prob_{i}^{(\theta)}\big[\xi_{1} = j,\,A_{1}<0,\,\vth_{1}=1\big]\ =\ \frac{1}{2}p_{ij}^{-}(\theta)\ >\ 0
	\end{align*}
	for all $i,j\in\cS$.
	
	\vspace{.2cm}
	(c) Finally, the same conclusions hold for the dual quantities. Indeed, since duality and geometric sampling commute (as noted at the start), and since all $\pi_{i}$ and $\pi_{i}(\theta)$ are positive, the same positivity results carry over to $\dualP^{*}$, $\dualP^{*}(\theta)$, and $\dualp_{ij}^{*\pm}(\theta)$.
\end{proof}

In order to show that geometric sampling is indeed compliant with our goal of proving Theorem~\ref{thm:main_{1}}, we must still verify that its assumptions \textsf{(B2--4)} remain valid for the sequence $(A_{n}^{*},B_{n}^{*})_{n\in\Z}$: 

\vspace{.1cm}
As noted above, the matrix $\sfP^{*}(\theta)$ has dominant eigenvalue
$$
\rho^{*}(\theta)\ =\ \frac{\rho(\theta)}{2 - \rho(\theta)} \quad \text{for all } \theta \in [0, \kappa],
$$
with derivative
$$
{\rho^{*}}'(\theta)\ =\ \frac{2 \rho'(\theta)}{(2 - \rho(\theta))^{2}}.
$$
Hence, by \eqref{eq:stationary drift}, the stationary drift of $(\xi_{n}^{*},S_{n}^{*})_{n\ge 0}$ under $\Prob^{(\kappa)}$ equals $2\rho'(\kappa)$, where $S_{n}^{*} := \log |\mathit{\Pi}_{n}^{*}|$. By assumption \textsf{(B2)}, we obtain
\begin{align*}
	&\sum_{i,j\in\cS}u_{i}(\kappa)v_{j}(\kappa)\,\Erw_{i}\Big[|A_{1}^{*}|^{\kappa}\log |A_{1}^{*}|\1_{\{\xi_{1}^{*}=j\}}\Big]\\
	&\hspace{2cm}=\ 2\rho'(\kappa)\ =\ 2\sum_{i,j\in\cS}u_{i}(\kappa)v_{j}(\kappa)\,\Erw_{i}\Big[|A_{1}|^{\kappa}\log |A_{1}|\1_{\{\xi_{1}=j\}}\Big]\ <\ \infty,
\end{align*}
and so $\Erw|A_{1}^{*}|^{\kappa} \log |A_{1}^{*}| < \infty$ follows from the positivity of all components of $u(\kappa)$ and $v(\kappa)$.

\smallskip
Turning to the finiteness of $\Erw|B_{1}^{*}|^{\kappa}$, we note:
\begin{align*}
	\Erw|B_{1}^{*}|^{\kappa}\ &=\ \Erw\bigg|\sum_{k=1}^{\vth_{1}}B_{k}A_{k+1}\cdots A_{\vth_{1}}\bigg|^{\kappa}\ \le\ \sum_{n\ge 1}\frac{1}{2^{n}}\sum_{k=1}^{n}\Erw|B_{k}A_{k+1}\cdots A_{n}|^{\kappa}\\
	&=\ \sum_{n\ge 1}\frac{1}{2^{n}}\sum_{k=1}^{n}\Erw|B_{0}A_{1}\cdots A_{n-k}|^{\kappa}\ =\ \sum_{n\ge 1}\frac{1}{2^{n}}\sum_{k=0}^{n-1}\Erw\Big[|B_{0}|^{\kappa}e^{\kappa S_{k}}\Big]\\
	&=\ \sum_{n\ge 1}\frac{1}{2^{n}}\sum_{k=0}^{n-1}\sum_{i\in\cS}\pi_{i}\,\Erw_{i}|B_{0}|^{\kappa}\,\Erw_{i}e^{\kappa S_{k}}\\
	&\le\ \Erw|B_{0}|^{\kappa}\Bigg(1+\sum_{n\ge 1}\frac{1}{2^{n}\|v(\kappa)\|_{0}}\sum_{k=1}^{n-1}\sum_{i\in\cS}\,\Erw_{i}\big[e^{\kappa S_{k}}v_{\xi_{k}}(\kappa)\big]\Bigg),\quad\|x\|_{0}\,:=\,\min_{i\in\cS}|x_{i}|\\
	&=\ \Erw|B_{0}|^{\kappa}\Bigg(1+\sum_{n\ge 1}\frac{n-1}{2^{n}\|v(\kappa)\|_{0}}\sum_{i\in\cS}v_{i}(\kappa)\Bigg)\ <\ \infty,
\end{align*}
In the fourth equality, we used the conditional independence of $B_{0}$ and $S_{k}$ given $\xi_{0}$, which follows from the Markov-modulated structure of the sequence $(\xi_{k}, A_{k}, B_{k})_{k\ge 0}$.
The last equality relies on the martingale property of the process $(e^{\kappa S_{n}} v_{\xi_{n}}(\kappa))_{n\ge 0}$ under $\Prob$, as discussed earlier.

\vspace{.1cm}
As explained in Subsection~\ref{subsec:lattice-type}, condition \textsf{(B3)} for $\zeta_{1}^{*} = \log |A_{1}^{*}|$ means that the associated \MRW\ $(\xi_{n}^{*}, S_{n}^{*})_{n \ge 0}$ is nonlattice under $\Prob^{(\kappa)}$. This follows from Lemma~A.6 in \cite{Alsmeyer:97} on the persistence of lattice type, because $(\xi_{n}, S_{n})_{n \ge 0}$ is itself nonlattice.

\vspace{.1cm}
Finally, condition \textsf{(B4)} for $B_{1}^{*}$ follows directly from
$$
\Prob[B_{1}^{*} = 0] \ \le\ \Prob[\vth_{1} = 1,\, B_{1} = 0] + \Prob[\vth_{1} > 1] \ =\ \frac{\Prob[B_{1} = 0] + 1}{2} \ <\ 1.
$$

\vspace{.2cm}
After these considerations, it suffices to prove Theorem~\ref{thm:main_{1}} under the \textbf{additional assumption}
\begin{itemize}[leftmargin=1.4cm]
	\item[\textsf{(B+)}] For any $\theta \in \bbD$, the entries of $(p_{ij}^{-}(\theta))_{i,j\in\cS}$, $(p_{ij}^{+}(\theta))_{i,j\in\cS}$, and $\sfP(\theta) = (\sfp_{ij}(\theta))_{i,j\in\cS}$ are strictly positive. The same holds for their dual counterparts.
\end{itemize}

\section{Subsampling at hitting times}

Keeping the starred notation from the previous section, suppose now that $(\vth_{n})_{n\in\Z}$, again satisfying
$$
\ldots < \vth_{-2} < \vth_{-1} < \vth_{0} = 0 < \vth_{1} < \vth_{2} < \ldots,
$$
denotes a sequence of hitting times associated with $(\xi_{n}, A_{n}, B_{n})_{n\in\Z}$. More precisely, for a given measurable set $E \subset \cS \times \R^2$, define the hitting times forward and backward in time by
$$
\vth_{n} := \inf\big\{k > \vth_{n-1} : (\xi_{\ell}, A_{\ell}, B_{\ell})_{\vth_{n-1} < \ell \le k} \in E\big\},
$$
and
$$
\vth_{-n} := \inf\big\{k < \vth_{-n+1} : (\xi_{\ell}, A_{\ell}, B_{\ell})_{k < \ell \le \vth_{-n+1}} \in E\big\},
$$
for all $n \in \N$. Assume that these hitting times are $\Prob$-almost surely finite. Since $\Prob$ dominates both $\Prob^{(\theta)}$ on $\cF_{\infty}$ and $\dProb^{(\theta)}$ on $\dualcF_{\infty}$ for every $\theta \in [0,\kappa]$, we then have
$$
\Prob^{(\theta)}[\vth_{n} < \infty] = 1 \quad\text{and}\quad \dProb^{(\theta)}[\vth_{-n} < \infty] = 1 \quad\text{for all } n \in \N.
$$

It follows that the subsampled sequences
$$
(\xi_{n}^{*})_{n \ge 0} := (\xi_{\vth_{n}})_{n \ge 0}
\quad \text{and} \quad
(\dualxi_{n}^{*})_{n \ge 0} := (\xi_{\vth_{-n}})_{n \ge 0}
$$
are both Markov chains on $\cS$. Since $\cS$ is finite, these chains are irreducible and hence positive recurrent on some subset $\cS^{*} \subset \cS$.

\section{Exponential smoothing}

Given a measurable function $f:\cS\times\R\to\R$ such that $f(i,\cdot)$ is Lebesgue integrable for each $i\in\cS$, we define $\ovl{f}:\cS\times\R\to\R$, called \emph{exponential smoothing of $f$}, as in \cite{Goldie:91} (when fixing $i$) by
$$
\ovl{f}(i,t)\ :=\ \int_{(-\infty,t]} \e^{-(t-x)} f(i,x)\,\diff x\ =\ \Erw \big[f(i,t-Z)\big],\quad i\in\cS,
$$
where $Z$ denotes a standard exponential random variable. The main properties of exponential smoothing are summarized in the following lemma.

\begin{Lemma}\label{lem:properties smoothing}
	Given $f,\ovl{f}$ as above, the following assertions hold:
	\begin{itemize}[leftmargin=1.5cm, itemsep=4pt]
		\item[\textsf{(ES-1)}] For each $i\in\cS$, the function $\ovl{f}(i,\cdot)$ is directly Riemann integrable, and
		$$
		\int_{\R} f(i,x)\,\diff x\ =\ \int_{\R} \ovl{f}(i,x)\,\diff x.
		$$
		
		\item[\textsf{(ES-2)}] If $f$ satisfies the Markov renewal equation $f = g + f*\Lambda$ for a bounded function $g$ and a finite kernel $\Lambda$ from $\cS$ to $\cS\times\R$, that is,
		$$
		f(i,t)\ =\ g(i,t)\ +\ \int_{\cS\times\R} f(s,t-x)\, \Lambda(i,\diff s\times\diff x)\quad\text{for all } (i,t)\in\cS\times\R,
		$$
		then so does $\ovl{f}$, namely
		$$
		\ovl{f}\,=\,\ovl{g} + \ovl{f*\Lambda}\,=\,\ovl{g} + \ovl{f} * \Lambda.
		$$
	\end{itemize}
\end{Lemma}

\begin{proof}
	\textsf{(ES-1)} is proved in \cite[Lemma 9.2]{Goldie:91}. To show \textsf{(ES-2)}, observe that
	\begin{align*}
		\ovl{f}(i,t)
		&=\ \Erw f(i,t-Z)\\
		&=\ \Erw g(i,t-Z)\ +\ \Erw\int_{\cS\times\R} f(s,t-x-Z)\ \Lambda(i,\diff s\times\diff x)\\
		&=\ \ovl{g}(i,t)\ +\ \int_{\cS\times\R} \Erw f(s,t-x-Z)\ \Lambda(i,\diff s\times\diff x)\\
		&=\ \ovl{g}(i,t)\ +\ \int_{\cS\times\R} \ovl{f}(s,t-x)\ \Lambda(i,\diff s\times\diff x).
	\end{align*}
	for all $(i,t)\in\cS\times\R$.
\end{proof}

It has been shown in \cite[Lemma 9.3]{Goldie:91} that, for any random variable $X$ and $\kappa > 0$, the implication
$$ \lim_{s\to\infty}\frac{1}{s}\int_{0}^{s} u^{\kappa}\,\Prob[X > u]\,\diff u\ =\ C
\quad\Longrightarrow\quad
\lim_{s\to\infty} s^{\kappa} \,\Prob[X > s]\ =\ C $$
holds true. The connection with exponential smoothing becomes apparent by observing that this implication is equivalent to
\begin{equation}\label{eq:convergence and smoothing}
	\lim_{t\to\infty}\ovl{f}(t)\ =\ C\quad\Longrightarrow\quad\lim_{t\to\infty}f(t)\ =\ C
\end{equation}
where $f(t) := \e^{\kappa t} \,\Prob[X > \e^{t}]$. In the present context, we apply this equivalence to the family of functions $f(i,t) := \e^{\kappa t}\, \Prob_{i}[X > \e^{t}]$, $i \in \cS$, with $C_{i}$ in place of $C$.

\section{Proof of Theorem \ref{thm:main_{1}}}\label{sec:proof thm main_{1}}

We distinguish the two cases
$$ \Prob[A_{1}\ge 0\hspace{.8pt}]=1\quad\text{and}\quad\Prob[A_{1}< 0\hspace{.8pt}]>0 $$
and will prove the second case by reduction to the first one. We also note that some technical lemmata regarding integrability and lattice-type are collected at the end of this section under auxiliaries. Recall that w.l.o.g. \textsf{(B+)} is assumed besides \textsf{(B1--4)}.

\vspace{.2cm}
\cblue\textsc{Case 1}. \cblack $\Prob[A_{1} \ge 0\,] = 1$. Then $S_{n} = \log \itPi_{n}$ and $\dualS_{n} = \log \itPi_{-n}$ for all $n \in \Z$. Under the assumptions of the theorem, we have $S_{n} \to -\infty$ a.s., and thus $\itPi_{n} \to 0$ a.s. by the strong law of large numbers (since $\Erw \log|A_{1}| < 0$). Consequently, $\itPi_{n} R_{n} \to 0$ in probability, by the stationarity of $(R_{n})_{n \ge 0}$ under $\Prob$. As already noted at the beginning of Subsection~\ref{subsec:lattice-type}, the \MRW\ $(\xi_{n}, S_{n})_{n \ge 0}$ is nonlattice under $\Prob^{(\kappa)}$ (by \textsf{(B3)}), and it has finite and positive stationary drift
$$
\rho'(\kappa) = \sum_{i} u_{i}(\kappa)\, \Erw_{i}[\e^{\kappa S_{1}} S_{1}]
$$
(by \textsf{(B2)}, \eqref{eq:stationary drift} and Lemma \ref{lem:Pdelta^kappa case 1}(f)). The same holds true for the dual process $(\dualxi_{n}, \dualS_{n})_{n \ge 0}$ under $\dProb^{(\kappa)}$.
Define
\begin{gather}
	G(i, t)\ :=\ \Prob_{i}[R_{0} > \e^{t}], \nonumber \\
	\Delta(i, t)\ :=\ \Prob_{i}[R_{0} > \e^{t}] - \Prob_{i}[A_{0} R_{-1} > \e^{t}]
	\ =\ \Prob_{i}[R_{0} > \e^{t}] - \Prob_{i}[R_{-1} > \e^{t - \dualS_{1}}], \nonumber \\
	I_{i}^{+} := \int_{0}^{\infty} \big(\Prob_{i}[R_{0} > t] - \Prob_{i}[A_{0} R_{-1} > t]\big) t^{\kappa - 1} \, \diff t\ = \int_{\R} \e^{\kappa t} \Delta(i, t) \, \diff t,\nonumber
	\shortintertext{and}
	\|I_{i}^{+}\|_{1} := \int_{0}^{\infty} \big| \Prob_{i}[R_{0} > t] - \Prob_{i}[A_{0} R_{-1} > t] \big| t^{\kappa - 1} \, \diff t\ = \int_{\R} \e^{\kappa t} |\Delta(i, t)| \, \diff t.\label{eq:def Iplus}
\end{gather}
It will be shown in Lemma \ref{lem:Iplus finite} that the assumptions of the theorem imply $I_{i}^{+}\le\|I_{i}^{+}\|_{1} < \infty$ for every $i\in\cS$.

We further note that, conditional on $\dualxi_{n} = \xi_{-n} = j$, the random variables $(\xi_{0}, \dualS_{n})$ and $R_{-n}$ are conditionally independent under $\Prob$, for any $j \in \cS$ and $n \in \N$. Moreover, the stationarity of $(\xi_{n}, R_{n})_{n \in \Z}$ implies $\Prob[R_{-n}\in\cdot\,|\dualxi_{n}=j]=\Prob_{j}[R_{0}\in\cdot]$. It follows that, for all $i\in\cS$ and $n\in\N$,
\begin{align*}
	\Prob_{i}\big[R_{-n}>\e^{t-\dualS_{n}}\big]\ &=\ \Prob\big[\xi_{0}=i,R_{-n}>\e^{t-\dualS_{n}}\big]/\pi_{i}\\
	&=\ \sum_{j\in\cS}\int_{\R}\Prob_{j}\big[R_{0}>\e^{t-x}\big]\ \Prob\big[\xi_{0}=i,\dualxi_{n}=j,\dualS_{n}\in\,\diff x\big]/\pi_{i}\\
	&=\ \sum_{j\in\cS}\int_{\R}\Prob_{j}\big[R_{0}>\e^{t-x}\big]\ \Prob_{i}\big[\dualxi_{n}=j,\dualS_{n}\in\,\diff x\big].
\end{align*}

Choosing $n = 1$, we see that $G(i, t)$ satisfies the Markov renewal equation
\begin{align*}
	G(i,t)\ &=\ \Delta(i,t)\ +\ \sum_{j\in\cS}\int_{\R}G(j,t-x)\ \Prob_{i}\big[\dualxi_{1}=j,\dualS_{1}\in\,\diff x\big]\\
	&=\ \Delta(i,t)\ +\ \Erw_{i}G(\dualxi_{1},t-\dualS_{1}).
\end{align*}
Defining
$$
g^{(\kappa)}(i,t)\,:=\,\frac{\e^{\kappa t}g(i,t)}{\dualv_{i}(\kappa)}\quad\text{for any function }g:\cS\times\R\to\R,
$$
multiplying the equation above by $\e^{\kappa t}/\dualv_{i}(\kappa)$ (i.e., applying exponential tilting) yields
\begin{align*}
	G^{(\kappa)}(i, t)\ =\ \Delta^{(\kappa)}(i, t)\ +\ \dErw_{i}^{(\kappa)}G^{(\kappa)}(\dualxi_{1},t-\dualS_{1}),
\end{align*}
and thus, after $n$ iterations for any $n\in\N$,
\begin{align*}
	G^{(\kappa)}(i,t)\ =\ \sum_{k=0}^{n-1}\dErw_{i}\Delta^{(\kappa)}(\dualxi_{k},t-\dualS_{k})\ +\ \dErw_{i}^{(\kappa)}G^{(\kappa)}(\dualxi_{n},t-\dualS_{n})
\end{align*}
for all $i\in\cS$ and $t\in\R$. By \textsf{(ES-2)} of Lemma~\ref{lem:properties smoothing}, this equation remains valid after exponential smoothing, that is,
\begin{align}\label{eq:renewal id for Gbar_kappa}
	\ovl{G^{(\kappa)}}(i,t)\ =\ \sum_{k=0}^{n-1}\dErw_{i}^{(\kappa)}\ovl{\Delta^{(\kappa)}}(\dualxi_{k},t-\dualS_{k})\ +\ \dErw_{i}^{(\kappa)}\ovl{G^{(\kappa)}}(\dualxi_{n},t-\dualS_{n}).
\end{align}
As $n\to\infty$, the last expectation vanishes for every $i$ and $t$, since $G$ is bounded and satisfies $\lim_{t \to \infty} \ovl{G}(i,t) = 0$ for each $i$, while $\lim_{n \to \infty} \dualS_{n} = -\infty$ $\Prob$-a.s. Moreover, we have
\begin{align}\label{eq:crucial remainder estimate}
	\begin{split}
		\dErw_{i}^{(\kappa)}\ovl{G^{(\kappa)}}(\dualxi_{n},t-\dualS_{n})\ &=\ \frac{\e^{\kappa t}}{\dualv_{i}(\kappa)}\,\Erw_{i}\Big[e^{-\kappa Z}G(\dualxi_{n},t-\dualS_{n}-Z)\Big]\\
		&\le\ \frac{\e^{\kappa t}}{\dualv_{i}(\kappa)}\,\Erw_{i}G(\dualxi_{n},t-\dualS_{n}-Z),
	\end{split}
\end{align}
where $Z$ denotes a standard exponential random variable independent of all other occurring random variables. Therefore, we conclude from \eqref{eq:renewal id for Gbar_kappa} that
\begin{align}\label{eq:crucial MRT formula}
	\ovl{G}^{(\kappa)}(i,t)\ =\ \sum_{n\ge 0}\dErw_{i}^{(\kappa)}\ovl{\Delta}^{(\kappa)}(\dualxi_{n},t-\dualS_{n}).
\end{align}
By Lemma~\ref{lem:Iplus finite}, the function $\Delta^{(\kappa)}(i,\cdot)$ is Lebesgue integrable over $\R$ for each $i\in\cS$. Hence, by property \textsf{(ES-2)} of Lemma~\ref{lem:properties smoothing}, its exponential smoothing $\ovl{\Delta}^{(\kappa)}(i,\cdot)$ is a directly Riemann integrable function with the same integral. Furthermore, the assumptions of the theorem guarantee that $(\dualxi_{n},\dualS_{n}){n\ge 0}$ is a nonlattice \MRW\ with positive drift $\rho'(\kappa)$ under $\dProb^{(\kappa)}_{i}$. Therefore, we can apply the discrete Markov renewal theorem as stated in \cite[Thm.~3.2]{Alsmeyer:14} to obtain
\begin{align*}
	\lim_{t\to\infty}\ovl{G^{(\kappa)}}(i,t)
	\ &=\ \frac{1}{\rho'(\kappa)}\sum_{j\in\cS}\pi_{j}(\kappa)\int_{\R}\Delta^{(\kappa)}(j,x)\diff x\\
	&=\ \frac{1}{\rho'(\kappa)}\sum_{j\in\cS}\frac{\pi_{j}(\kappa)}{\dualv_{j}(\kappa)}\int_{\R}\e^{\kappa x}\Delta(j,x)\diff x\ =\ \frac{1}{\rho'(\kappa)}\sum_{j\in\cS}\frac{\pi_{j}(\kappa)}{\dualv_{j}(\kappa)}I_{j}^{+},
\end{align*}
or, equivalently,
\begin{equation*}
	\lim_{t\to\infty}\e^{\kappa t}\Prob_{i}[R_{0}>\e^{t}]
	\ =\ \frac{\dualv_{i}(\kappa)}{\rho'(\kappa)}\sum_{j\in\cS}\frac{\pi_{j}(\kappa)}{\dualv_{j}(\kappa)}I_{j}^{+}
	\ =:\ \Ciplus
\end{equation*}
for each $i\in\cS$. By equation~\eqref{eq:Iplusminus expectation} in Lemma~\ref{lem:Iplus finite}, the constant $\Ciplus$ can also be expressed as
\begin{equation}\label{eq:def Ciplus}
	\Ciplus\ =\ \frac{\dualv_{i}(\kappa)}{\rho'(\kappa)}\sum_{j\in\cS}\frac{\pi_{j}(\kappa)}{{\dualv_{j}(\kappa)}}\Erw_{j}\Big[(R_{0}^{+})^{\kappa}-((A_{0}R_{-1})^{+})^{\kappa}\Big].
\end{equation}
This proves the right-tail part of~\eqref{eq1:left-right tail nu} and, as a direct consequence,
\begin{align}\label{eq3:right tail nu}
	\lim_{t\to\infty}t^{\kappa}\,\Prob_{\pi}[R_{0}>t]
	\ &=\ \sum_{i\in\cS}\pi_{i}\lim_{t\to\infty}t^{\kappa}\,\Prob_{i}[R_{0}>t]
	\ =\ \sum_{i\in\cS}\pi_{i}\Ciplus
	\ =:\ \Cplus.
\end{align}
Next, define
\begin{gather}
	I_{i}^{-}\ :=\ \int_{0}^{\infty}\big(\Prob_{i}[R_{0}<-t] \,-\, \Prob_{i}[A_{0}R_{-1}<-t]\big)t^{\kappa-1}\,\diff t,\nonumber
	\shortintertext{and}
	\|I_{i}^{-}\|_{1}\ :=\ \int_{0}^{\infty}\big|\Prob_{i}[R_{0}<-t] \,-\, \Prob_{i}[A_{0}R_{-1}<-t]\big|t^{\kappa-1}\,\diff t,
	\label{eq:def Iminus}
	\intertext{which are finite by another appeal to Lemma~\ref{lem:Iplus finite}, and}
	\Ciminus\ :=\ \frac{\dualv_{i}(\kappa)}{\rho'(\kappa)}
	\sum_{j\in\cS}\frac{\pi_{j}(\kappa)}{\dualv_{j}(\kappa)}I_{j}^{-}
	\ =\ \frac{\dualv_{i}(\kappa)}{\rho'(\kappa)}
	\sum_{j\in\cS}\frac{\pi_{j}(\kappa)}{\dualv_{j}(\kappa)}
	\Erw_{j}\Big[(R_{0}^{-})^{\kappa}-((A_{0}R_{-1})^{-})^{\kappa}\Big].
	\label{eq:def Ciminus}
\end{gather}
Then the corresponding statement in~\eqref{eq1:left-right tail nu} 
for the left tails of $\Prob_{i}[R_{0}\in\cdot]$, and a similar statement as \eqref{eq3:right tail nu} for $\Prob[R_{0}\in\cdot]$, with limiting constants $\Ciminus$ and $\Cminus:=\sum_{i\in\cS}\pi_{i}\Ciminus$, follow in the same manner. Indeed, replacing $R_{n}$ by $-R_{n}$ amounts to replacing $B_{n}$ by $-B_{n}$ while keeping $A_{n}$ unchanged.

\vspace{.2cm}
\cblue\textsc{Case 2}. \cblack $\Prob[A_{1}<0\hspace{.8pt}]>0$. 
Naturally, we can restrict our attention to the right tail of $R_{0}$ here as well. 
The main idea of the proof is to reduce this situation to the first case by comparing $R_{0}$ with 
$\dualPi_{\sigma}\dualR_{\sigma} \,=\, \itPi_{-\sigma}R_{-\sigma}$, where
\begin{equation}\label{eq:sign switch}
	\sigma\ :=\ \inf\{n\ge 1:\Pi_{-n}\ge 0\}\ =\ 
	\begin{cases}
		\hfill 1&\text{if }A_{0}\ge 0,\\
		\inf\{n\ge 2:A_{-n+1}\le 0\}&\text{otherwise}.
	\end{cases}
\end{equation}
Introducing the bivariate ``sign chain'' 
$(\dualxi_{n},\sign(\dualPi_{n}))_{n\ge 0}
= (\xi_{-n},\sign(\itPi_{-n}))_{n\ge 0}$ 
with state space $\cS\times\{-1,0,+1\}$, 
we observe that $\sigma$ is the first hitting time of $\cS\times\{0,+1\}$ and is $\Prob$-a.s.~finite, 
since $\Prob_{i}[\xi_{-1}=j,\,A_{0}<0\hspace{.8pt}]>0$ for some $i,j\in\cS$, 
and $(\xi_{-n})_{n\ge 0}$ is recurrent.

We note that under $\dProb^{(\kappa)}$, the sign $\delta=0$ does not occur almost surely, since
$$
\dProb_{i}^{(\kappa)}[A_{1}=0]\,=\, \frac{1}{v_{i}(\kappa)}\,\Erw_{i}\!\big[|A_{1}|^{\kappa}v_{\xi_{1}}(\kappa)\1_{\{A_{1}=0\}}\big]
\,=\, 0
\quad\text{for every }i\in\cS.
$$
Therefore, the state space reduces to $\cS\times\{-1,+1\}$, and $\sigma$ becomes the first hitting time of $\cS\times\{+1\}$. 
By \textsf{(B+)}, the sign chain is irreducible and positive recurrent on $\cS\times\{-1,+1\}$, 
with transition probabilities
$$
\dProb^{(\kappa)}\!\big[(\xi_{-n},\sign(\itPi_{-n}))=(j,\pm\delta)\,\big|\,(\xi_{-n+1},\sign(\itPi_{-n+1}))=(i,\delta)\big]
\,=\, \dualp_{ij}^{\,\pm}(\kappa)$$
for $\delta\in\{-1,+1\}$ and $n\ge 1$, 
where (see~\eqref{eq:def pplusminus})
\begin{equation}\label{eq:def dual pplusminus}
	\dualp_{ij}^{\,-}(\kappa)\,:=\,\dProb_{i}^{(\kappa)}[\xi_{-1}=j,\,A_{0}<0\hspace{.8pt}],
	\quad\text{and}\quad
	\dualp_{ij}^{\,+}(\kappa)\,:=\,\dProb_{i}^{(\kappa)}[\xi_{-1}=j,\,A_{0}>0\hspace{.8pt}].
\end{equation}
Thus, transitions from $(i,\delta)$ to $(j,\pm\delta)$ depend only on $i\in\cS$.

\vspace{.1cm}
Denote by $(\pi_{i,\delta}(\kappa))_{i\in\cS,\delta=\pm 1}$ the stationary law of the sign chain, 
and note that $\pi_{i}(\kappa)=\pi_{i,-1}(\kappa)+\pi_{i,+1}(\kappa)$. 
Positive recurrence ensures $\dProb^{(\kappa)}[\sigma<\infty]=1$ and, in particular,
\begin{equation}\label{eq:Ehat sigma finite}
	\dErw_{i}^{(\kappa)}\sigma\,<\,\infty 
	\quad\text{for all } i\in\cS,
\end{equation}
a fact we will use repeatedly.

\vspace{.1cm}
Further defining $\sigma_{0}:=0$, $\sigma_{1}:=\sigma$, and
$$ \sigma_{n}\,:=\,\inf\{k>\sigma_{n-1}:A_{-\sigma_{n-1}}\cdots A_{-k+1}\ge 0\},
\quad n\ge 2, $$
the sequence $(\dualxi_{n}^{*},\dualS_{n}^{*}):=(\xi_{-\sigma_{n}},S_{-\sigma_{n}})$ for $n\ge 0$ 
constitutes a \MRW\ under $\dProb^{(\kappa)}$. 
Let $\pi^{*}(\kappa)=(\pi_{i}^{*}(\kappa))_{i\in\cS}$ denote the stationary distribution of 
its positive recurrent driving chain $(\dualxi_{n}^{*})_{n\ge 0}$ under $\dProb^{(\kappa)}$, 
and observe that $(\dualxi_{n}^{*},\dualR_{n}^{*})_{n\ge 0}$, with
$$
\dualR_{n}^{*}\ :=\ \dualR_{\sigma_{n}}
\,=\, \sum_{k\ge 0}\dualA_{\sigma_{n}}\cdots\dualA_{\sigma_{n}-k+1}\dualB_{\sigma_{n}-k}
\,=\, \sum_{k\ge 0}A_{-\sigma_{n}}\cdots A_{-\sigma_{n}+k-1}B_{-\sigma_{n}+k}
\,=\, R_{-\sigma_{n}},
$$
is stationary under $\dProb_{\pi^{*}(\kappa)}^{(\kappa)}$. 
Since, furthermore,
$$
\dualR_{n}^{*}\ =\ \dualPsi_{n}^{*}(\dualR_{n-1}^{*})
\quad\text{with}\quad
\dualPsi_{n}^{*}\,:=\,\dualPsi_{\sigma_{n-1}}\cdots\dualPsi_{\sigma_{n}-1},
$$
we see that $(\dualxi_{n}^{*},\dualR_{n}^{*})_{n\ge 0}$ is of the same type as 
$(\dualxi_{n},\dualR_{n})_{n\ge 0}$, obtained by subsampling, 
but with associated affine maps $\dualPsi_{n}^{*}$ having positive slope 
$\dProb_{\pi^{*}(\kappa)}^{(\kappa)}$-a.s. 
It is shown in the next lemma that $\pi^{*}(\kappa)$ and $\pi(\kappa)$ actually coincide, 
so that $\dProb_{\pi^{*}(\kappa)}^{(\kappa)}=\dProb_{\pi(\kappa)}^{(\kappa)}$.

\begin{Lemma}\label{lem:signed chain stationary law}
	Under the stated assumptions,
	\begin{equation}\label{eq:signed chain stationary law}
		\pi_{i,-}(\kappa)\ =\ \pi_{i,+}(\kappa)\ =\ \frac{\pi_{i}(\kappa)}{2}\ =\ \frac{\pi_{i}^{*}(\kappa)}{2}
	\end{equation}
	for all $i\in\cS$. In particular, $\pi(\kappa)=\pi^{*}(\kappa)$. Moreover, $\dErw_{\pi(\kappa)}^{(\kappa)}\sigma=2$.
\end{Lemma}

\begin{proof}
	We first note that the law of
	$$ \sigma(i)\ :=\ \inf\{n\ge 1:\xi_{-n}=i\text{ and }\sign(\itPi_{-n})=1\} $$
	is identical under $\dProb_{(i,-1)}^{(\kappa)}$ and $\dProb_{(i,+1)}^{(\kappa)}$, 
	since transitions of the sign chain from any $(i,\delta)$ to any other state in 
	$\cS\times\{-1,+1\}$ depend only on $i$. 
	Observe also that $\sigma(i)$ equals the return time to $(i,\delta)$ when starting in this state, 
	for each $\delta\in\{\pm1\}$. 
	Consequently, the return times of the states $(i,+1)$ and $(i,-1)$ under $\dProb_{i}^{(\kappa)}$ coincide, yielding in particular
	$$ \frac{1}{\dErw_{i}^{(\kappa)}[\sigma(i)]}\ =\ \pi_{i,-}(\kappa)\ =\ \pi_{i,+}(\kappa)\ =\ \frac{\pi_{i}(\kappa)}{2}. $$
	For the final equality in~\eqref{eq:signed chain stationary law}, note that 
	$(\sigma_{n})_{n\ge 1}$ forms a stationary and ergodic sequence under 
	$\dProb_{\pi^{*}(\kappa)}^{(\kappa)}$. 
	Hence, by Birkhoff's ergodic theorem, 
	$n^{-1}\sigma_{n}\to\dErw_{\pi^{*}(\kappa)}^{(\kappa)}\sigma=:\sfm$ a.s.~under this measure. 
	Applying the ergodic theorem for Markov chains, we obtain
	\begin{align*}
		\pi_{i}^{*}(\kappa)
		&=\ \lim_{n\to\infty}\frac{1}{n}\sum_{k=1}^{n}\1_{\{i\}}(\xi_{-\sigma_{k}})\\
		&=\ \lim_{n\to\infty}\frac{1}{n}\sum_{k=1}^{\sigma_{n}}\1_{\{i\}\times\{+1\}}(\xi_{-k},\sign(\itPi_{-k}))\\
		&=\ \lim_{n\to\infty}\frac{\sigma_{n}}{n}\cdot
		\lim_{n\to\infty}\frac{1}{\sigma_{n}}\sum_{k=1}^{\sigma_{n}}\1_{\{i\}\times\{+1\}}(\xi_{-k},\sign(\itPi_{-k}))\\
		&=\ \sfm\cdot\pi_{i,+}(\kappa)\ =\ \frac{\sfm}{2}\,\pi_{i}(\kappa)
		\quad\dProb_{\pi^{*}(\kappa)}^{(\kappa)}\text{-a.s.}
	\end{align*}
	for all $i\in\cS$. 
	Summing over all $i$ on both sides gives $\sfm=2$, and hence $\pi(\kappa)=\pi^{*}(\kappa)$.
\end{proof}

\vspace{.1cm}
After these considerations, the announced reduction of Case~2 to Case~1 -- by comparing $R_{0}$ with $\dualPi_{\sigma}\dualR_{\sigma}=\itPi_{-\sigma}R_{-\sigma}$ rather than with $(\dualPi_{1},\dualR_{1})$ -- amounts to working with the sequence $(\dualxi_{n}^{*},\dualR_{n}^{*})_{n\ge 0}$, whose associated \MRW\ is $(\dualxi_{n}^{*},\dualS_{n}^{*})_{n\ge 0}$, instead of $(\dualxi_{n},\dualR_{n})_{n\ge 0}$.

\vspace{.1cm}
To be precise, we define
\begin{gather}
	\Delta_{\sigma}(i,t)\,:=\,\Prob_{i}[R_{0}>\e^{t}] \,-\, \Prob_{i}[\itPi_{-\sigma}R_{-\sigma}>\e^{t}]
	\,=\, \Prob_{i}[R_{0}>\e^{t}] \,-\, \Prob_{i}[R_{-\sigma}>\e^{t-\dualS_{\sigma}}],\nonumber\\
	I_{\sigma,i}^{+}\ :=\ \int_{0}^{\infty}\big(\Prob_{i}[R_{0}>t] \,-\, \Prob_{i}[\itPi_{-\sigma}R_{-\sigma}>t]\big)t^{\kappa-1}\,\diff t
	\ =\ \int_{\R}\e^{\kappa t}\Delta_{\sigma}(i,t)\,\diff t,\label{eq:def Isigmaplus}
	\shortintertext{and}
	I_{\sigma,i}^{-}\ :=\ \int_{0}^{\infty}\big(\Prob_{i}[R_{0}<-t] \,-\, \Prob_{i}[\itPi_{-\sigma}R_{-\sigma}<-t]\big)t^{\kappa-1}\,\diff t.\label{eq:def Isigmaminus}
\end{gather}
Moreover, let $\|I_{\sigma,i}^{+}\|_{1}$ and $\|I_{\sigma,i}^{-}\|_{1}$ be the obvious counterparts of $\|I_{i}^{+}\|_{1}$ and $\|I_{i}^{-}\|_{1}$ defined in \eqref{eq:def Iplus} and \eqref{eq:def Iminus}, respectively. Proceeding as in Case~1, we find that $G(i,t)$ satisfies the Markov renewal equation
$$ G(i,t)\ =\ \Delta_{\sigma}(i,t)\ +\ \Erw_{i}\!\big[G(\dualxi_{1}^{*},t-\dualS_{1}^{*})\big], $$
which after exponential tilting, smoothing and $n$ iterations takes the form
\begin{align}\label{eq:renewal id for Gbar_kappa, Case 2}
	\ovl{G^{(\kappa)}}(i,t)\ =\ \sum_{k=0}^{n-1}\dErw_{i}^{(\kappa)}\ovl{\Delta_{\sigma}^{(\kappa)}}(\dualxi_{k}^{*},t-\dualS_{k}^{*})\ +\ \dErw_{i}^{(\kappa)}\ovl{G^{(\kappa)}}(\dualxi_{n}^{*},t-\dualS_{n}^{*}).
\end{align}
for all $i\in\cS$ and $t\in\R$. The last expectation vanishes as $n\to\infty$, since the inequality \eqref{eq:crucial remainder estimate} remains valid with $(\dualxi_{n}^{*},\dualS_{n}^{*})$ in the place of $(\dualxi_{n},\dualS_{n})$. Thus we obtain
\begin{align}\label{eq:crucial MRT formula Case 2}
	\ovl{G^{(\kappa)}}(i,t)\ =\ \sum_{n\ge 0}\dErw_{i}^{(\kappa)}\ovl{\Delta_{\sigma}^{(\kappa)}}(\dualxi_{n}^{*},t-\dualS_{n}^{*}).
\end{align}
which serves as our analogue of~\eqref{eq:crucial MRT formula} in the present case. 

\vspace{.1cm}
To obtain the (right-tail part of)~\eqref{eq2:left-right tail nu} with the asserted constant 
$(\Ciminus+\Ciplus)/2$ via another application of the Markov renewal theorem, we must verify the finiteness of $\|I_{\sigma,i}^{\pm}\|_{1}$ and the necessary properties for the new \MRW\ $(\dualxi_{n}^{*},\dualS_{n}^{*})_{n\ge 0}$ under $\dProb^{(\kappa)}$ 
(positive drift, lattice type). This is done in the two Lemmata~\ref{lem:sampling at positive sign} and \ref{lem:Isigma finite} in the auxiliary section. With these results at hand, we infer from~\eqref{eq:crucial MRT formula Case 2} that
$$ \lim_{t\to\infty}\ovl{G^{(\kappa)}}(i,t)\ =\ \frac{1}{\dErw^{(\kappa)}[\dualS_{1}^{*}]}
\sum_{j\in\cS}\pi_{j}(\kappa)\int_{\R}\Delta_{\sigma}^{(\kappa)}(j,x)\,\diff x
\ =\ \frac{1}{2\rho'(\kappa)}\sum_{j\in\cS}\frac{\pi_{j}(\kappa)}{\dualv_{j}(\kappa)}I_{\sigma,j}^{+},
$$
which is equivalent to
\begin{equation}\label{eq:formula C}
	\lim_{t\to\infty}t^{\kappa}\,\Prob_{i}[R_{0}>t]
	\ =\ \frac{\dualv_{i}(\kappa)}{2\rho'(\kappa)}\sum_{j\in\cS}\frac{\pi_{j}(\kappa)}{\dualv_{j}(\kappa)}I_{\sigma,j}^{+}\ =:\ C_{i},
\end{equation}
for each $i\in\cS$, recalling that $t^{\kappa}\,\Prob_{i}[R_{0}>t]=\dualv_{i}(\kappa)G^{(\kappa)}(i,t)$ and the limits of $G^{(\kappa)}(i,t)$ and $\ovl{G^{(\kappa)}}(i,t)$ coincide (see \eqref{eq:convergence and smoothing}).  
It thus only remains to verify that $C_{i}=(\Ciminus+\Ciplus)/2$, 
which in fact amounts to computing the $I_{\sigma,i}^{\pm}$ in terms of the $I_{j}^{\pm}$.

\vspace{.1cm}
To this end, we start from~\eqref{eq:Isigmaplusminus expectation} for $I_{\sigma,i}^{+}$ and obtain
\begin{align*}
	I_{\sigma,i}^{+}\ &=\ \Erw_{i}\Big[(R_{0}^{+})^{\kappa}-((\Pi_{-\sigma}R_{-\sigma})^{+})^{\kappa}\Big]\\
	&=\ \sum_{k\ge 0}\sum_{n=0}^{k-1}\Erw_{i}\Big[\big(((\Pi_{-n}R_{-n})^{+})^{\kappa}-((\Pi_{-n-1}R_{-n-1})^{+})^{\kappa}\big)\1_{\{\sigma=k\}}\Big]\\
	&=\ \Erw_{i}\Big[(R_{0}^{+})^{\kappa}-((A_{0}R_{-1})^{+})^{\kappa}\Big]+\sum_{n\ge 1}\Erw_{i}\Big[\big(((\Pi_{-n}R_{-n})^{+})^{\kappa}-((\Pi_{-n-1}R_{-n-1})^{+})^{\kappa}\big)\1_{\{\sigma=k\}}\Big]\\
	&=\ I_{i}^{+}\,+\,\sum_{n\ge 1}\Erw_{i}\Big[|\Pi_{-n}|^{\kappa}\,\Erw\big[(R_{-n}^{-})^{\kappa}-((A_{-n}R_{-n-1})^{-})^{\kappa}|\xi_{-n}\big]\1_{\{\sigma>n\}}\Big]\\
	&=\ I_{i}^{+}\,+\,\sum_{n\ge 1}\Erw_{i}\Big[e^{\kappa\dualS_{n}}I_{\dualxi_{n}}^{-}\1_{\{\sigma>n\}}\Big]\ =\ I_{i}^{+}\,-\,I_{i}^{-}\,+\,\Erw_{i}\Bigg[\sum_{n=0}^{\sigma-1}e^{\kappa\dualS_{n}}I_{\dualxi_{n}}^{-}\Bigg]
\end{align*}
For the penultimate line, observe that $\itPi_{-n}<0$ on the event $\{\sigma>n\}$ implies
\begin{gather*}
	(\Pi_{-n}R_{-n})^{+}\ =\ |\Pi_{-n}|R_{-n}^{-}\quad\text{and}\quad (\Pi_{-n-1}R_{-n-1})^{+}\ =\ |\Pi_{-n}|(A_{-n-1}R_{-n})^{-}
\end{gather*}
for each $n\ge 1$. Returning to~\eqref{eq:formula C}, the proof of~\eqref{eq2:left-right tail nu} is now completed by the following calculation:
\begin{align*}
	C_{i}\ &=\ \frac{\dualv_{i}(\kappa)}{2\rho'(\kappa)}\sum_{j\in\cS}\frac{\pi_{j}(\kappa)}{\dualv_{j}(\kappa)}I_{\sigma,j}^{+}\\
	&=\ \frac{\dualv_{i}(\kappa)}{2\rho'(\kappa)}\sum_{j\in\cS}\frac{\pi_{j}(\kappa)}{\dualv_{j}(\kappa)}\Big(I_{j}^{+}-I_{j}^{-}\Big)\,+\,\frac{\dualv_{i}(\kappa)}{2\rho'(\kappa)}\sum_{j\in\cS}\frac{\pi_{j}(\kappa)}{\dualv_{j}(\kappa)}\,\Erw_{j}\Bigg[\sum_{n=0}^{\sigma-1}\e^{\kappa\dualS_{n}}I_{\dualxi_{n}}^{-}\Bigg]\\
	&=\ \frac{\Ciplus-\Ciminus}{2}\,+\,\frac{\dualv_{i}(\kappa)}{2\rho'(\kappa)}\sum_{j\in\cS}\pi_{j}(\kappa)\,\dErw_{j}^{(\kappa)}\Bigg[\sum_{n=0}^{\sigma-1}\frac{1}{\dualv_{\dualxi_{n}}(\kappa)}I_{\dualxi_{n}}^{-}\Bigg]\\
	&=\ \frac{\Ciplus-\Ciminus}{2}\,+\,\frac{\dualv_{i}(\kappa)}{\rho'(\kappa)}\,\frac{1}{\dErw_{\pi(\kappa)}^{(\kappa)}\sigma}\,\dErw_{\pi(\kappa)}^{(\kappa)}\Bigg[\sum_{n=0}^{\sigma-1}\frac{1}{\dualv_{\dualxi_{n}}(\kappa)}I_{\dualxi_{n}}^{-}\Bigg]\\
	&=\ \frac{\Ciplus-\Ciminus}{2}\,+\,\frac{\dualv_{i}(\kappa)}{\rho'(\kappa)}\,\dErw_{\pi(\kappa)}^{(\kappa)}\Bigg[\frac{1}{\dualv_{\dualxi_{0}}(\kappa)}I_{\dualxi_{0}}^{-}\Bigg]\ =\ \frac{\Ciplus-\Ciminus}{2}\,+\,\frac{\dualv_{i}(\kappa)}{\rho'(\kappa)}\sum_{j\in\cS}\frac{\pi_{j}(\kappa)}{\dualv_{j}(\kappa)}I_{j}^{-}\\
	&=\ \frac{\Ciplus-\Ciminus}{2}\,+\,\Ciminus\ =\ \frac{\Ciplus+\Ciminus}{2}.
\end{align*}
To go from the fourth to the fifth line, the occupation measure formula \eqref{eq:occupation measure formula} stated in the Appendix is used. It provides
\begin{equation*}
	\big({\dErw_{\pi(\kappa)}^{(\kappa)}\sigma}\big)^{-1}\,\dErw_{\pi(\kappa)}^{(\kappa)}\Bigg[\sum_{n=0}^{\sigma-1}\frac{1}{\dualv_{\dualxi_{n}}(\kappa)}I_{\dualxi_{n}}^{-}\Bigg]\ =\ \dErw_{\pi(\kappa)}^{(\kappa)}\Bigg[\frac{1}{\dualv_{\dualxi_{1}}(\kappa)}I_{\dualxi_{1}}^{-}\Bigg]
\end{equation*}
and relies on the two facts that the laws of $\dualxi_{0}$ and $\dualxi_{\sigma}$ coincide under $\dProb_{\pi^{*}(\theta)}^{(\theta)}=\dProb_{\pi(\theta)}^{(\theta)}$.

\vspace{.2cm}
We note that, by \eqref{eq:def Ciplus} and \eqref{eq:def Ciminus}, 
\begin{equation}\label{eq:def Ciplus+Ciminus}
	\frac{\Ciplus+\Ciminus}{2}\ =\ \frac{\dualv_{i}(\kappa)}{2\rho'(\kappa)}\sum_{j\in\cS}\frac{\pi_{j}(\kappa)}{{\dualv_{j}(\kappa)}}\Erw_{j}\Big[|R_{0}|^{\kappa}-|A_{0}R_{-1}|^{\kappa}\Big].
\end{equation}

\subsection*{Positivity of $\Ciplus+\Ciminus$}
To complete the proof of Theorem~\ref{thm:main_{1}}, we finally show that $\Ciplus+\Ciminus>0$ for each $i\in\cS$ if $(\xi_{n},R_{n})_{n\ge 0}$ is nondegenerate. This will be done by identifying $\Prob_{i}[R_{0}\in\cdot]$ for each $i\in\cS$ as the stationary law of a nondegenerate -- in the sense of \eqref{nondegenerate iid} -- random affine recursion with i.i.d.~coefficients and then using the known result in this case.

\vspace{.1cm}
For each $i\in\cS$, let $(\tau_{n}(i))_{n\in\Z}$ be the increasing sequence of return times to $i$ of the doubly infinite stationary Markov chain $(\xi_{n})_{n\in\Z}$, with the convention that
$$ \tau_{0}(i)\,\le\,0\,<\,\tau_{1}(i), $$
thus $\tau(i):=\tau_{1}(i)=\inf\{n\ge 1:\xi_{n}=i\}$. Put $\sfm_{i}:=\Erw_{i}\tau(i)$. Further define $R_{n}^{(i)}:=R_{\tau_{n}(i)}$ and observe that
\begin{equation*}
	R_{n}^{(i)}\,=\,A_{n}^{(i)}R_{n-1}^{(i)}+B_{n}^{(i)}
\end{equation*}
for every $n\in\Z$, where
\begin{gather*}
	A_{n}^{(i)}\,:=\,\prod_{k=\tau_{n-1}(i)+1}^{\tau_{n}(i)}A_{k}\quad\text{and}\quad B_{n}^{(i)}\,:=\,\sum_{k=\tau_{n-1}(i)+1}^{\tau_{n}(i)}\bigg(\prod_{l=k+1}^{\tau_{n}(i)}A_{l}\bigg)B_{k}.
\end{gather*}
Under $\Prob_{i}$, i.e. when $\tau_{0}(i)=0$, the $(A_{n}^{(i)},B_{n}^{(i)})$ are i.i.d.~and, as $(R_n)_{n\geq 0}$ is stationary, also $(R_{n}^{(i)})_{n\ge 0}$ is a stationary sequence with
\begin{gather}\label{eq:stationary Rzero^i}
	R_{0}^{(i)}\ =\ 
	B_{0}^{(i)}+A_{0}^{(i)}B_{-1}^{(i)}+A_{0}^{(i)}A_{-1}^{(i)}B_{-2}^{(i)}+\ldots\ =\ R_{0},
\end{gather}
that is, the stationary law of $(R_{n}^{(i)})_{n\ge 0}$ is the law of $R_{0}$ under $\Prob_{i}$.
With this observation, the convergence of
$$ t^{\kappa}\,\Prob_{i}[R_{0}^{(i)}>t]\quad\text{and}\quad t^{\kappa}\,\Prob_{i}[R_{0}^{(i)}<-t] $$
can also be inferred from the Goldie-Kesten theorem for systems with i.i.d.~coefficients, provided that
\begin{gather}
	\Erw_{i}\big|A_{1}^{(i)}\big|^{\kappa}\,=\,1,\label{Ai1}\\
	\Prob\big[\log|A_{1}^{(i)}|\in d\Z\,|\,A_{1}^{(i)}\ne 0\big]\,<\,1\quad\text{for any }d>0,\label{Ai2}
	\shortintertext{and}
	\Erw\Big[\big|A_{1}^{(i)}\big|^{\kappa}\log|A_{1}^{(i)}|\Big]<\infty,\quad\text{and}\quad\Erw\Big[\big|B_{1}^{(i)}\big|^{\kappa}\Big]<\infty\label{Ai3}
\end{gather}
are valid under the assumptions of Theorem \ref{thm:main_{1}}. But Condition \eqref{Ai1} holds because
\begin{gather*}
	\Erw_{i}\big|A_{1}^{(i)}\big|^{\kappa}\ =\ \Erw_{i}\e^{\kappa S_{\tau(i)}}\ =\ \Erw_{i}\big[\e^{\kappa S_{\tau(i)}}\1_{\{\tau(i)<\infty\}}\big]\ =\ \dProb_{i}^{(\kappa)}[\tau(i)<\infty]\ =\ 1,
	\intertext{and \eqref{Ai3} follows from}
	\begin{split}
		\Erw_{i}\Big[\big|A_{1}^{(i)}\big|^{\kappa}\log|A_{1}^{(i)}|\Big]\ &=\ \Erw_{i}\big[\e^{\kappa S_{\tau(i)}}|S_{\tau(i)}|\big]\ =\ \dErw_{i}^{(\kappa)}|S_{\tau(i)}|\ \le\ \dErw_{i}^{(\kappa)}\Bigg[\sum_{k=1}^{\tau(i)}|X_{k}|\Bigg]\\
		&=\ \dErw_{i}^{(\kappa)}[\tau(i)]\,\dErw_{\pi}^{(\kappa)}|S_{1}|\ =\ \dErw_{i}^{(\kappa)}[\tau(i)]\,\Erw_{i}|A_{1}|^{\kappa}\log|A_{1}|\ <\ \infty,
	\end{split}
\end{gather*}
where the occupation measure formula \eqref{eq:occupation measure formula} was used for the penultimate equality. Lastly, the lattice-type condition \eqref{Ai2} is obtained by another appeal to Lemma~A.6 in \cite{Alsmeyer:97}.

\vspace{.1cm}
To finally infer
$$ \Ciplus+\Ciminus\ =\ \lim_{t\to\infty}t^{\kappa}\Big(\Prob_{i}[R_{0}^{(i)}>t]\,+\,\Prob_{i}[R_{0}^{(i)}<-t]\Big)\ >\ 0 $$
from Theorems 2.4.4 and 2.4.7 in \cite{BurDamMik:16},
we must verify that the $(A_{n}^{(i)},B_{n}^{(i)})_{n\ge 1}$ satisfy \eqref{nondegenerate iid}, for every $i\in\cS$. But this is a direct consequence of the nondegeneracy of $(\xi_{n},R_{n})_{n\ge 0}$ and the following lemma which has been shown in \cite[Lemma 4.1 and Prop.\,4.6]{AlsBuck:17a}.

\begin{Lemma}\label{lem:Buckmann's lemma}
	For an \MMLIFS\ $(\xi_{n},R_{n})_{n\ge 0}$ satisfying the assumptions of Theorem \ref{thm:main_{1}}, the following assertions are equivalent:
	\begin{itemize}\itemsep1pt
		\item[(a)] $(\xi_{n},R_{n})_{n\ge 0}$ is nondegenerate.
		\item[(b)] $\Prob[A_{1}^{(i)}c+B_{1}^{(i)}=c\,]<1$ for all $c\in\R$ and all $i\in\cS$.
	\end{itemize}
\end{Lemma}

\subsection*{Auxiliaries} 

The following lemma goes back to Goldie~\cite[Lemma~9.4]{Goldie:91}, where 
\eqref{eq:integration of abs tail differences} appeared as an identity, which is not true in general (the corrected form is stated e.g.~in \cite[p.\,53/54]{BurDamMik:16}).

\begin{Lemma}\label{lem:Goldie's lemma}
	Let $X,Y$ be real-valued random variables and $\kappa>0$. Then
	\begin{equation}\label{eq:integration of abs tail differences}
		\int_{0}^{\infty}\!\big|\Prob[X>t]-\Prob[Y>t]\big|\,t^{\kappa-1}\,\diff t\ \le\ \frac{1}{\kappa}\,\Erw\Big|(X^{+})^{\kappa}-(Y^{+})^{\kappa}\Big|.
	\end{equation}
	If the last expectation is finite, then absolute value signs can be removed to give
	\begin{equation}\label{eq:integration of tail differences}
		\int_{0}^{\infty}\!\big(\Prob[X>t]-\Prob[Y>t]\big)\,t^{\kappa-1}\,\diff t\ =\ \frac{1}{\kappa}\,\Erw\Big[(X^{+})^{\kappa}-(Y^{+})^{\kappa}\Big].
	\end{equation}
\end{Lemma}

Using this result, we now establish two integrability lemmata that serve as fundamental auxiliaries for the proof of Theorem~\ref{thm:main_{1}}. 
Observe that $R_{0}=A_{0}R_{-1}+B_{0}$ implies
\begin{equation}\label{eq:estimate 1}
	\big|R_{0}^{\pm}-(A_{0}R_{-1})^{\pm}\big|\ \le\ |B_{0}|.
\end{equation}More generally, if $\sigma$ denotes an arbitrary integer-valued random variable, then 
$R_{0}=\itPi_{-\sigma}R_{-\sigma}+\sum_{k=0}^{\sigma-1}\itPi_{-k}B_{-k}$, and by the same reasoning,
\begin{equation}\label{eq:estimate 2}
	\Big|R_{0}^{\pm}-(\itPi_{-\sigma}R_{-\sigma})^{\pm}\Big|\ \le\ \Bigg|\sum_{k=0}^{\sigma-1}\itPi_{-k}B_{-k}\Bigg|\ =\ \Bigg|\sum_{k=0}^{\sigma-1}\e^{\dualS_{k}}\dualB_{k}\Bigg|.
\end{equation}
This bound will be used in Lemma~\ref{lem:Isigma finite} below for the stopping time $\sigma$ defined in~\eqref{eq:sign switch}.

\begin{Lemma}\label{lem:Iplus finite}
	Under the basic assumptions of Theorem~\ref{thm:main_{1}}, together with conditions \textsf{(B1)} and \textsf{(B2)}, 
	the integrals $\|I_{i}^{+}\|_{1}$ and $\|I_{i}^{-}\|_{1}$, defined in 
	\eqref{eq:def Iplus} and~\eqref{eq:def Iminus}, are finite for every $i\in\cS$. 
	Consequently, the same holds for the integrals $I_{i}^{\pm}$, which moreover satisfy
	\begin{gather}\label{eq:Iplusminus expectation}
		I_{i}^{\pm}\ =\ \Erw_{i}\!\Big[(R_{0}^{\pm})^{\kappa}-((A_{0}R_{-1})^{\pm})^{\kappa}\Big].
	\end{gather}
\end{Lemma}

\begin{proof}
	Define
	$$ \mu_{\theta}(A)\ :=\ \max_{i,j\in\cS}\Erw\big[|A_{0}|^{\theta}\,\big|\,\xi_{0}=i,\,\xi_{1}=j\big],
	\qquad \theta>0, $$
	and analogously for $\mu_{\theta}(B)$. 
	Note that these quantities are finite for all $\theta\in[0,\kappa]$. 
	For $\theta\in(0,\kappa)$, it follows from Lemma~\ref{lem:Ptheta vs Lambda} that
	$$ \frac{1}{n}\log\Erw\big[\e^{\theta S_{n}}\big]\,\le\,\log\!\left(\frac{1+\rho(\theta)}{2}\right)
	\,<\,0,\quad\text{for all $n$ sufficiently large.} $$
	Hence, for $\theta\in(0,\kappa\wedge1)$, the perpetuity $R_{0}:=\sum_{n\ge 0}\itPi_{-n}B_{-n}=\sum_{n\ge 0}\e^{S_{-n}}B_{-n}$ satisfies, by subadditivity,
	$$ \Erw|R_{0}|^{\theta}\,\le\,\mu_{\theta}(B)\sum_{n\ge 0}\Erw|\itPi_{-n}|^{\theta}
	\,=\,\mu_{\theta}(B)\sum_{n\ge 0}\Erw\big[\e^{\theta S_{n}}\big]\,<\,\infty. $$
	For $\theta\in[1,\kappa)$ with $\kappa>1$, the same conclusion follows from Minkowski's inequality for the $L_{\theta}$-norm of~$R_{0}$.
	
	\vspace{.1cm}
	Using the inequality
	$$ (x+y)^{\theta}\ \le\ x^{\theta}+\theta\,2^{\theta-1}\big(x^{\theta-1}y+xy^{\theta-1}\big)+y^{\theta},
	\qquad x,y\ge0, $$
	valid for $\theta>1$, and the subadditivity of $x\mapsto x^{\theta}$ for $\theta\in(0,1)$, 
	estimate~\eqref{eq:estimate 1} yields
	\begin{align}\label{eq:estimate 1b}
		\begin{split}
			&(R_{0}^{\pm})^{\theta}-((A_{0}R_{-1})^{\pm})^{\theta}\ \le\ ((A_{0}R_{-1})^{\pm}+|B_{0}|)^{\theta}-((A_{0}R_{-1})^{\pm})^{\theta}\\
			&\le\ 
			\begin{cases}
				|B_{0}|^{\theta},&\text{if }\theta\in (0,1],\\[1mm]
				|B_{0}|^{\theta}+c_{\theta}\big(((A_{0}R_{-1})^{\pm})^{\theta-1}|B_{0}|+(A_{0}R_{-1})^{\pm}|B_{0}|^{\theta-1}\big),&\text{if }\theta >1,
			\end{cases}
		\end{split}
	\end{align}
	where $c_{\theta}:=\theta\,2^{\theta-1}$. 
	From this bound and the assumptions of the lemma, we deduce the finiteness of 
	$\Erw\big|(R_{0}^{\pm})^{\kappa}-((A_{0}R_{-1})^{\pm})^{\kappa}\big|$, since
	\begin{align*}
		\Erw&\big|(R_{0}^{\pm})^{\kappa}-((A_{0}R_{-1})^{\pm})^{\kappa}\big|\\
		&\le\ 
		\begin{cases}
			\Erw\big[|B_{0}|^{\kappa}\big],&\text{if }\kappa\in (0,1],\\[1mm]
			\Erw|B_{0}|^{\kappa}
			+c_{\kappa}\Big(\mu_{\kappa-1}(A)\mu_{1}(B)\Erw|R_{0}|^{\kappa-1}
			+\mu_{1}(A)\mu_{\kappa-1}(B)\Erw|R_{0}|\Big),&\text{if }\kappa>1.
		\end{cases}
	\end{align*}
	Finally, applying \eqref{eq:integration of abs tail differences} and \eqref{eq:integration of tail differences} from Lemma~\ref{lem:Goldie's lemma} yields $\|I_{i}^{\pm}\|_{1}<\infty$ and hence identity~\eqref{eq:Iplusminus expectation}.
\end{proof}

\begin{Lemma}\label{lem:Isigma finite}
	Under the basic assumptions of Theorem~\ref{thm:main_{1}} together with $\Prob[A_{1}<0]>0$ and conditions \textsf{(B1)}, \textsf{(B2)}, 
	let $\sigma$ denote the stopping time defined by~\eqref{eq:sign switch}. Then
	$$
	\Erw\Bigg|\sum_{k=0}^{\sigma-1}\dualA_{k}\dualB_{k}\Bigg|^{\kappa}
	\ =\ 
	\Erw\Bigg|\sum_{k=0}^{\sigma-1}\itPi_{-k}B_{-k}\Bigg|^{\kappa}
	\ <\ \infty.
	$$
	As a consequence, the integrals $I_{\sigma,i}^{+}$ and $I_{\sigma,i}^{-}$, defined in 
	\eqref{eq:def Isigmaplus} and~\eqref{eq:def Isigmaminus}, respectively, are both finite and satisfy
	\begin{gather}\label{eq:Isigmaplusminus expectation}
		I_{\sigma,i}^{\pm}\ =\ 
		\Erw_{i}\big[(R_{0}^{\pm})^{\kappa}\big]
		\,-\,\Erw_{i}\big[((\itPi_{-\sigma}R_{-\sigma})^{\pm})^{\kappa}\big]
	\end{gather}
	for every $i\in\cS$.
\end{Lemma}

\begin{proof}
	In view of \eqref{eq:estimate 2} and Lemma~\ref{lem:Goldie's lemma}, it suffices to establish the first assertion.  
	This will be done by induction over $n$, where the induction parameter is chosen such that $\kappa\in(2^{\,n-1},2^{\,n}]$.  
	Throughout the proof, $\gamma$ denotes a generic positive constant that may change from line to line.  
	Recall from~\eqref{eq:Ehat sigma finite} that $\dErw^{(\kappa)}\sigma<\infty$.
	
	\vspace{.1cm}
	\textsc{Base case $n=0$.}  
	If $\kappa\in(\frac{1}{2},1]$, then subadditivity together with
	$$ \sup_{i\in\cS}\Erw\big[|\dualB_{k}|^{\kappa}\big|\dualxi_{k-1}=i\big]\ \le\ \gamma\,\Erw\big[|\dualB_{1}|^{\kappa}\big]\ =\ \gamma\,\Erw\big[|B_{1}|^{\kappa}\big] $$
	yields
	\begin{align*}
		\Erw\Bigg|\sum_{k=0}^{\sigma-1}\dualA_{k}\dualB_{k}\Bigg|^{\kappa}\ &=\ \Erw\Bigg|\sum_{k\ge 0}\e^{\dualS_{k}}\dualB_{k}\1_{\{\sigma>k\}}\Bigg|^{\kappa}\ \le\ \sum_{k\ge 0}\Erw \big[\e^{\kappa\dualS_{k}}|\dualB_{k}|^{\kappa}\1_{\{\sigma>k\}}\big]\\
		&\le\ \sum_{k\ge 0}\Erw \big[\e^{\dualS_{k}}|\dualB_{k}|^{\kappa}\1_{\{\sigma>k\}}\big]\ =\ \sum_{k\ge 0}\Erw\Big[\e^{\dualS_{k}}\,\Erw\big[|\dualB_{k}|^{\kappa}|\dualxi_{k-1}\big]\1_{\{\sigma>k\}}\Big]\\
		&\le\ \gamma\,\Erw|B_{1}|^{\kappa}\sum_{k\ge 0}\Erw \Big[\e^{\dualS_{k}}\1_{\{\sigma>k\}}\Big]\ \le\ \gamma\,\Erw|B_{1}|^{\kappa}\sum_{k\ge 0}\dProb^{(\kappa)}[\sigma>k]\\
		&\le\ \gamma\,\Erw|B_{1}|^{\kappa}\,\dErw^{(\kappa)}\sigma\ <\ \infty.
	\end{align*}
	Obviously, the argument is valid for any $\kappa\in (0,1]$.
	
	\vspace{.1cm}
	\textsc{Inductive step.}  
	Let $\kappa\in(2^{\,n},2^{\,n+1}]$ for some $n\ge0$.  
	Assume the lemma holds for all $\kappa'\le 2^{\,n}$ and every Markov-modulated sequence $(A_{k},B'_{k})_{k\in\Z}$ satisfying the same assumptions.
	
	Define
	$$
	\wh\sfm_{k}\ :=\ \Erw\big[\dualB_{k}\,\big|\,\dualxi_{k-1},\dualxi_{k}\big]
	\ =\ 
	\Erw\big[B_{-k}\,\big|\,\xi_{-k+1},\xi_{-k}\big]
	\ =:\ \sfm_{-k}, \qquad k\in\Z.
	$$
	Then
	$$
	\sum_{k\ge 0}\e^{\dualS_{k}}\dualB_{k}\1_{\{\sigma>k\}}
	\ =\ 
	\sum_{k\ge 0}\e^{\dualS_{k}}\wh\sfm_{k}\1_{\{\sigma>k\}}
	\,+\,
	\sum_{k\ge 0}\e^{\dualS_{k}}(\dualB_{k}-\wh\sfm_{k})\1_{\{\sigma>k\}}.
	$$
	The first sum has finite expectation since, using again the finiteness of~$\cS$,
	$$
	\sum_{k\ge 0}\Erw\big[\e^{\dualS_{k}}\wh\sfm_{k}\1_{\{\sigma>k\}}\big]
	\ \le\ \gamma\sum_{k\ge 0}\Erw\big[\e^{\dualS_{k}}\1_{\{\sigma>k\}}\big]
	\ \le\ \gamma\sum_{k\ge 0}\dProb^{(\kappa)}[\sigma>k]
	\ =\ \gamma\,\dErw^{(\kappa)}\sigma\ <\ \infty.
	$$
	The second sum is the a.s.~limit of the zero-mean martingale
	$$
	M_{n}\ :=\ \sum_{k=0}^{n-1}\e^{\dualS_{k}}(\dualB_{k}-\wh\sfm_{k})\1_{\{\sigma>k\}},
	\qquad n\ge1.
	$$
	Applying Burkholder's inequality (see, e.g., \cite[Thm.\,2.10]{Hall+Heyde:80}) gives
	$$
	\Erw\Bigg|\sum_{k\ge 0}\e^{\dualS_{k}}(\dualB_{k}-\wh\sfm_{k})\1_{\{\sigma>k\}}\Bigg|^{\kappa}
	\ \le\ \gamma\,\Erw\Bigg(\sum_{k\ge 0}\e^{2\dualS_{k}}(\dualB_{k}-\wh\sfm_{k})^{2}\1_{\{\sigma>k\}}\Bigg)^{\kappa/2}.
	$$
	By the inductive hypothesis, the expectation on the right-hand side is finite, since the sequence $(A_{k},(B_{k}-\sfm_{k})^{2})_{k\in\Z}$ satisfies the same assumptions with exponent $\kappa/2\le 2^{\,n}$.  
	This completes the induction and hence the proof.
\end{proof}

\begin{Lemma}\label{lem:sampling at positive sign}
	Under the assumptions of Theorem \ref{thm:main_{1}} including \textsf{(B+)}, the Cram\'er transform of $(\dualxi_{n}^{*},\dualzeta_{n}^{*})_{n\ge 0}$, viz.
	$$
	\dualsfP^{*}(\theta)\,:=\,\Big(\Erw_{i}\big[\e^{\theta\dualS_{1}^{*}}\1_{\{\dualxi_{1}^{*}=j\}}\big]\Big)_{i,j\in\cS}
	\,=\,\Big(\Erw_{i}\big[\e^{\theta S_{-\sigma}}\1_{\{\xi_{-\sigma}=j\}}\big]\Big)_{i,j\in\cS},
	$$
	is positive for each $\theta\in[0,\kappa]$. For $\theta=\kappa$, it has dominant eigenvalue $\rho^{*}(\kappa)=1$ with associated left and right eigenvectors $\dualu^{\,*}(\kappa)=\pi(\kappa)^{\top}\dualD(\kappa)^{-1}$ and $\dualv^{\,*}(\kappa)=\dualv(\kappa)$, respectively, satisfying $\dualu^{\,*}(\kappa)^{\top}\dualv^{\,*}(\kappa)=1$. Furthermore, recalling $\dErw_{\pi(\kappa)}^{(\kappa)}\sigma=2$,
	\begin{gather}
		\dErw_{i}^{(\kappa)}(\dualS_{1}^{*})^{+}\,=\,\frac{1}{\dualv_{i}(\kappa)}\,\Erw_{i}\Big[\e^{\kappa S_{-\sigma}}S_{-\sigma}^{+}\,\dualv_{\xi_{-\sigma}}(\kappa)\Big]\,<\,\infty,\label{eq:LlogL for Ssigma}\\
		\dErw_{\pi(\kappa)}^{(\kappa)}\dualS_{1}^{*}\,=\,2\rho'(\kappa)\ \in\ \IRg,\label{eq:drift embedded MRW}
		\shortintertext{and}
		\dProb_{\pi(\kappa)}^{(\kappa)}\big[S_{-\sigma}-a_{\xi_{-\sigma}}+a_{\xi_{0}}\in d\Z\big]\,<\,1,
		\label{eq:lattice-type S_sigma}
	\end{gather} 
	for any $d>0$ and $\{a_{i}:i\in\cS\}\subset[0,d)$, so $(\dualxi_{n}^{*},\dualS_{n}^{*})_{n\ge 0}$ is nonlattice under $\dProb^{(\kappa)}$. 
\end{Lemma}

\begin{proof}
	The positivity of $\dualsfP^{*}(\theta)$ for $\theta\in[0,\kappa]$ follows from \textsf{(B+)}, since
	$$
	\Erw_{i}\big[\e^{\theta S_{-\sigma}}\1_{\{\xi_{-\sigma}=j\}}\big]
	\,\ge\,\Erw_{i}\big[\e^{\theta S_{-1}}\1_{\{\sigma=1,\ \xi_{-1}=j\}}\big]
	\,=\,\dualsfp_{ij}^{+}(\theta)\,>\,0
	$$
	for all $i,j\in\cS$. For finiteness, use \eqref{eq:Ehat sigma finite} to obtain
	$$
	1\,=\,\dProb_{i}^{(\kappa)}[\sigma<\infty]
	\,=\,\sum_{n\ge1}\frac{1}{\dualv_{i}(\kappa)}\Erw_{i}\big[\e^{\kappa S_{-n}}\dualv_{\xi_{-n}}(\kappa)\1_{\{\sigma=n\}}\big]
	\,=\,\frac{1}{\dualv_{i}(\kappa)}\Erw_{i}\big[\e^{\kappa S_{-\sigma}}\dualv_{\xi_{-\sigma}}(\kappa)\big],
	$$
	whence $\dualsfP^{*}(\kappa)$ has dominant eigenvalue $\rho^{*}(\kappa)=1$ with right eigenvector $\dualv(\kappa)$. Consequently,
	$$
	\dualP^{*}(\kappa)\,=\,\dualD(\kappa)^{-1}\,\sfP^{*}(\kappa)\,\dualD(\kappa),
	$$
	so the normalized left eigenvector of $\sfP^{*}(\kappa)$ is $\dualu^{\,*}(\kappa)=\pi(\kappa)^{\top}\dualD(\kappa)^{-1}$, and $\dualu^{\,*}(\kappa)^{\top}\dualv^{\,*}(\kappa)=1$ (see also the end of Subsection~\ref{subsec:Cramer transform}).
	
	\vspace{.1cm}
	Next, by \eqref{eq:Ehat sigma finite} and the occupation measure formula \eqref{eq:occupation measure formula},
	$$
	\dErw_{\pi(\kappa)}^{(\kappa)}S_{-\sigma}^{+}\,\le\,\dErw_{\pi(\kappa)}^{(\kappa)}\!\left[\sum_{k=1}^{\sigma}\zeta_{-k}^{+}\right]\,=\,\dErw_{\pi(\kappa)}^{(\kappa)}\zeta_{-1}^{+}\ \dErw_{\pi(\kappa)}^{(\kappa)}\sigma
	\,<\,\infty,
	$$
	and therefore, for any $i,j\in\cS$,
	$$
	\Erw_{i}\big[\e^{\kappa S_{-\sigma}}S_{-\sigma}^{+}\1_{\{\xi_{-\sigma}=j\}}\big]
	\,=\,\frac{\dualv_{i}(\kappa)}{\dualv_{j}(\kappa)}\,\dErw^{(\kappa)}S_{-\sigma}^{+}
	\,<\,\infty.
	$$
	This proves \eqref{eq:LlogL for Ssigma}. Then
	$$
	\dErw_{\pi(\kappa)}^{(\kappa)}\dualS_{1}^{*}\,=\,\dErw_{\pi(\kappa)}^{(\kappa)}\!\Bigg[\sum_{n=1}^{\sigma}\dualzeta_{n}\Bigg]\,=\,\dErw_{\pi(\kappa)}^{(\kappa)}\dualS_{1}\ \dErw_{\pi(\kappa)}^{(\kappa)}\sigma
	\,=\,\rho'(\kappa)\,\dErw_{\pi(\kappa)}^{(\kappa)}\sigma\,=\,2\rho'(\kappa),
	$$
	which yields \eqref{eq:drift embedded MRW}.
	
	\vspace{.1cm}
	For the lattice-type statement \eqref{eq:lattice-type S_sigma}, assume to the contrary that for some $d>0$ (w.l.o.g.\ $d=1$) and $a_i\in[0,1)$,
	$$
	\dProb^{(\kappa)}\big[S_{-\sigma}-a_{\xi_{-\sigma}}+a_{\xi_{0}}\in\Z\big]\,=\,1.
	$$
	Put $\bfi=\sqrt{-1}$ and define, for $i,j\in\cS$ and $t\in\R$,
	\begin{gather*}
		\vph_{ij}^{\pm}(t)\,:=\,\dErw_{i}^{(\kappa)}\big[\e^{\bfi t\log|A_{0}|}\,\big|\,\xi_{-1}=j,\ A_{0}\gtrless0\big]\quad\text{and}\quad
		\vph_{ij}(t)\,:=\,\dErw_{i}^{(\kappa)}\big[\e^{\bfi t\log|A_{0}|}\,\big|\,\xi_{-1}=j\big].
	\end{gather*}
	Recall from \eqref{eq:def dual pplusminus} that $\dualp_{ij}^{\,\pm}(\kappa)=\dProb_{i}^{(\kappa)}\big[\xi_{-1}=j,\ A_{0}\gtrless 0\hspace{.8pt}\big]$.
	
	\vspace{.1cm}
	Our assumption implies, for all $i,j\in\cS$,
	$$
	\vph_{ij}^{+}(2\pi)
	\,=\,\Erw_{i}\big[\e^{2\pi\bfi\zeta_{-1}}\,\big|\,\xi_{-1}=j,\ A_{0}>0\big]
	\,=\,\Erw_{i}\big[\e^{2\pi\bfi S_{-\sigma}}\,\big|\,\xi_{-1}=j,\ A_{0}>0\big]
	\,=\,\e^{2\pi\bfi(a_{j}-a_{i})}.
	$$
	By \textsf{(B+)}, all $\dualp_{ij}(\kappa)$ and $\dualp_{ij}^{\,\pm}(\kappa)$ are positive.  
	For $\vph_{ij}^{-}(2\pi)$ we use the events $\{\xi_{-1}=j,\ \xi_{-2}=j,\ \sigma=2\}$ with $A_{0}<0$, $A_{-1}<0$ to obtain
	$$
	\e^{2\pi\bfi(a_{j}-a_{i})}
	\,=\,\dErw_{i}^{(\kappa)}\big[\e^{2\pi\bfi S_{-\sigma}}\,\big|\,\xi_{-1}=j,\ \xi_{-2}=j,\ A_{0}<0,\ A_{-1}<0\big]
	\,=\,\vph_{ij}^{-}(2\pi)\,\vph_{jj}^{-}(2\pi),
	$$
	and similarly
	$$ 1\,=\,\dErw_{j}^{(\kappa)}\big[\e^{2\pi\bfi S_{-\sigma}}\,\big|\,\xi_{-1}=j,\ \xi_{-2}=j,\ A_{0}<0,\ A_{-1}<0\big]\,=\,\vph_{jj}^{-}(2\pi)^{2}. $$
	Hence, for all $i,j$,
	$$ \vph_{ij}^{-}(2\pi)\,=\,\e^{2\pi\bfi(a_{j}-a_{i})}\quad\text{or}\quad\vph_{ij}^{-}(2\pi)\,=\,-\,\e^{2\pi\bfi(a_{j}-a_{i})}. $$
	Combining this with the formula for $\vph_{ij}^{+}(2\pi)$ yields
	$$ \dProb_{i}^{(\kappa)}\big[S_{-1}-a_{j}+a_{i}\in \Z/2\,\big|\,\xi_{-1}=j\big]\,=\,1
	\qquad\text{for all }i,j\in\cS, $$
	which contradicts \textsf{(B3)}. Therefore \eqref{eq:lattice-type S_sigma} holds, and the chain $(\dualxi_{n}^{*},\dualS_{n}^{*})$ is nonlattice under $\dProb^{(\kappa)}$.
\end{proof}

\section{Proof of Theorem \ref{thm:main_2}}\label{sec:proof thm main_2}

Consider the jump epochs of $J$,
$$
T_{0}\,:=\,0,
\quad\text{and}\quad
T_{n}\,:=\,\inf\{t>T_{n-1}: J_{t-}\neq J_{t}\}\quad\text{for }n\ge 1.
$$
From \eqref{MMGOUexplicit} it follows that
$$
V_{T_{n}}
\ =\ \e^{-(\zeta_{T_{n}}-\zeta_{T_{n-1}})}\,V_{T_{n-1}}
\,+\,\e^{-\zeta_{T_{n}}}\!\int_{(T_{n-1},T_{n}]}\e^{\zeta_{s-}}\,\diff\eta_s,
\qquad n\in\N.
$$
Hence, $(\xi_{n},R_{n}):=(J_{T_{n}},V_{T_{n}})$ is an \MMLIFS\ of the form 
$R_{n}=A_{n}R_{n-1}+B_{n}$,
with coefficients
\begin{equation}\label{eq-MMRREcoeff}
	(A_{n},B_{n})
	\ =\
	\bigg(\e^{-(\zeta_{T_{n}}-\zeta_{T_{n-1}})}\,,\,
	\int_{(T_{n-1},T_{n}]}\e^{-(\zeta_{T_{n}}-\zeta_{s-})}\,\diff\eta_s\bigg),
	\qquad n\ge 1.
\end{equation}

By replacing $(J,(\zeta,\eta))$ with its dual counterpart $(J^{*},(\zeta^{*},\eta^{*}))$, we can extend the above \MMLIFS\ by setting  $(\xi_{-n+1},A_{-n+1},B_{-n})_{n\ge 1}:=(\xi_{n}^{*},A_{n}^{*},B_{n}^{*})_{n\ge 1}$. This yields a doubly infinite stationary sequence $(\xi_{n},A_{n},B_{n})_{n\in\ZZ}$ with $\Prob[A_{1}\ge 0]=1$. We may thus apply Theorem \ref{thm:main_{1}} once conditions \textsf{(B1)} -- \textsf{(B4)} are checked.

\vspace{.2cm}
By definition \eqref{eq_upsilon},
$$
\Erw_{i}\Big[|A_{1}|^{\kappa}\,\1_{\{\xi_{1}=j\}}\Big]
\ =\ \Erw_{i}\!\left[\e^{-\kappa\zeta_{T_{1}}}\,\1_{\{J_{T_{1}}=j\}}\right]
\,=\,\Upsilon_{\zeta}^{ij}(\kappa),
$$
and therefore Assumption \textsf{(B1)} of Theorem~\ref{thm:main_{1}} is equivalent to Assumption \textsf{(A1)}.

\vspace{.2cm}
Moreover,
$$
|A_{1}|^{\kappa}\log|A_{1}|
\ =\ -\,\zeta_{T_{1}}\,\e^{-\kappa\zeta_{T_{1}}}
\ =\ \frac{\partial}{\partial\theta}\,\e^{-\theta\zeta_{T_{1}}}\Big|_{\theta=\kappa},
$$
and the first part of Assumption \textsf{(B2)} of Theorem~\ref{thm:main_{1}} holds in the present setting if and only if
$$
\frac{\partial}{\partial\theta}\Erw_{\pi}\big[\e^{-\theta\zeta_{T_{1}}}\big]\Big|_{\theta=\kappa}
\ =\ \sum_{i\in\cS}\sum_{j\in\cS}\pi_{i}\,
\frac{\partial}{\partial\theta}\Erw_{i}\!\big[\e^{-\theta\zeta_{T_{1}}}\,\1_{\{J_{T_{1}}=j\}}\big]\Big|_{\theta=\kappa}
\,<\,\infty,
$$
which is precisely condition~\textsf{(A2)}.

\vspace{.2cm}
To verify the second moment condition in \textsf{(B2)}, we apply \cite[Thm.~3.5]{BDTS}, noting that $T_{1}$ is exponentially distributed and therefore possesses moments of all orders. 
Assume first that $\kappa\ge 1$, in which case $\Erw_{\pi}|\eta_{1}|^{\kappa}<\infty$ follows from \textsf{(A3)}. 
Moreover, for any $\varepsilon>0$,
\begin{align*}
	&\Erw_{\pi}\!\left[\sup_{s\le T_{1}}\big|\e^{-(\zeta_{T_{1}}-\zeta_{s-})}\big|^{\kappa+\varepsilon}\right]<\infty\\
	&\qquad\Longleftrightarrow\quad
	\Erw_{j}\!\left[\sup_{s\le T_{1}}\big|\e^{-(\zeta_{T_{1}}-\zeta_{s-})}\big|^{\kappa+\varepsilon}\right]<\infty\quad\text{for all }j\in\cS\\
	&\qquad\Longleftrightarrow\quad
	\Erw_{j}\!\left[\e^{-(\kappa+\varepsilon)\Delta\zeta_{T_{1}}}\,
	\sup_{s\le T_{1}}\big|\e^{-(\zeta^{(j)}_{T_{1}-}-\zeta^{(j)}_{s-})}\big|^{\kappa+\varepsilon}\right]<\infty
	\quad\text{for all }j\in\cS\\
	&\qquad\Longleftrightarrow\quad
	\Erw_{j}\!\left[\e^{-(\kappa+\varepsilon)\Delta\zeta_{T_{1}}}\right]\,
	\Erw_{j}\!\left[\sup_{s\le T_{1}}\e^{-(\kappa+\varepsilon)\zeta^{(j)}_{s}}\right]<\infty
	\quad\text{for all }j\in\cS.
\end{align*}
Finiteness of the two factors in the last line for each $j$ and for some suitable $\varepsilon$ is ensured by \textsf{(A4)}. 
Hence, by \cite[Thm.~3.5]{BDTS}, we obtain $\Erw_{\pi}|B_{1}|^{\kappa}<\infty$. 
If $\kappa<1$, the same argument with $\kappa'=1$ yields $\Erw_{\pi}|B_{1}|^{\kappa}\le\Erw_{\pi}|B_{1}|<\infty$.

\vspace{.2cm}
Concerning \textsf{(B3)}, we observe that
$$\Prob\big[-\zeta_{T_{1}} - a_{J_{T_{1}}} + a_{J_{T_{0}}}\in d\ZZ\big]\,<\,1\quad\text{for all }d\in\R\text{ and }\{a_k:k\in\cS\}\subset\R,$$
holds if, for some pair $(i,j)\in\cS^2$,
$$\Prob\big[-\zeta^{(j)}_{T_{1}-} - Z^{ji}_{\zeta,1} -a_{i}+a_{j} \in d\ZZ\big]\,<\,1.$$
By independence of $-\zeta^{(j)}_{T_{1}-}$ and $Z^{ji}_{\zeta,1}$, this condition is clearly satisfied if either $\zeta^{(j)}$ has infinite variation (see \cite[Cor.~24.6]{sato2nd}) or if $\zeta^{(j)}$ has nonzero drift. In the case of pure-jump L\'evy processes $\zeta^{(j)}, j\in\cS$, the stated assumption that the jump sizes of some $\zeta^{(j)}$ are nonlattice, or that $Z^{ji}_{\zeta,1}$ has a nonlattice distribution, is sufficient for our purposes.\\
Lastly, since $\eta \not\equiv 0$, either there exists $j\in\cS$ such that $\eta^{(j)}\not\equiv 0$, or there exists a pair $(j,i)$ such that a transition of $J$ from $j$ to $i$ induces a nonzero jump in $\eta$ with positive probability. Both scenarios imply, for this $j$,
\begin{align*} 
	\Prob_{j}[B_{1}=0]
	\ &=\ \Prob_{j}\left[\e^{-\zeta_{T_{1}}}\int_{(0,T_{1})} \e^{\zeta_{s-}}\,\diff\eta^{(j)}_s + \e^{-\Delta \zeta_{T_{1}}} \Delta\eta_{T_{1}}\,=\,0 \right]\ <\ 1,
\end{align*}
as the jump size $\Delta \eta_{T_{1}}$ is independent of $T_{1}$ and of $(\zeta^{(j)},\eta^{(j)})$. Therefore, $\Prob_\pi[B_{1}=0]=\sum_{j\in\cS}  \pi_{j}\Prob_{j}[B_{1}=0]<1$, and Assumption \textsf{(B4)} of Theorem~\ref{thm:main_{1}} is satisfied.\\
Thus, Assumptions \textsf{(B1)}--\textsf{(B4)} of Theorem~\ref{thm:main_{1}} hold, and we conclude that
\begin{equation} \label{eq-zetatoinfty}
	\Erw\zeta_{T_{1}}\,=\,- \Erw \log |A_{1}|\,>\,0,
\end{equation}
and that the stationary law $\nu$ of $(R_{n})$ has a power-law decay as stated in \eqref{eq1:left-right tail nu}. In particular, positivity of the sum of the constants follows, since degeneracy of the considered \MMLIFS~is equivalent to \eqref{eq-MMGOUdegenerate}, see \cite[Section 4.3]{BS_ExpFuncMAP}, which has been ruled out by assumption.  

\vspace{.1cm}
To show that $(V_{t})_{t\ge 0}$ admits a nontrivial stationary distribution, we recall from \cite[Thm.~3.3]{BSidMMGOU} that it suffices to prove that 
$\int_{(0,t]} \e^{\wh{\zeta}_{s-}}\,\diff \wh{L}_s$ converges in $\wh{\Prob}_{\pi}$-probability, as $t\to\infty$, to a finite-valued random variable. As $\cS$ is finite, this is equivalent to $\wh{\Prob}_\pi$-a.s. convergence by \cite[Rem.~4.2]{BS_ExpFuncMAP}. To prove the latter, note first that, by a generalization of Wald's equation (see, e.g., \cite[Cor.~XI.2.6]{Asmussen:03}), \eqref{eq-zetatoinfty} implies
$$ \Erw\zeta_{1}\,=\,\frac{\Erw\zeta_{T_{1}}}{\Erw T_{1}}\,>\,0, $$
and hence $\lim_{t\to \infty} \zeta_{t}=\infty$ $\Prob_\pi$-a.s. Together with \textsf{(A4)} we can therefore conclude that the long-term mean $\kappa_\zeta$ of $\zeta$ (see, e.g., \cite[Eq.~(3.6)]{BS_ExpFuncMAP} for its explicit definition) is positive and finite. From this, we immediately obtain for the dual process that $\lim_{t\to \infty} -\wh{\zeta}_{t}=\infty$ $\wh{\Prob}_{\pi}$-a.s. and $0<\kappa_{-\wh{\zeta}}<\infty$. Thus, by \cite[Prop.~5.2 and 5.7.1]{BS_ExpFuncMAP}, we obtain the desired convergence of $\int_{(0,t]} \e^{\wh{\zeta}_{s-}}\,\diff \wh{L}_s$ provided that
\begin{equation} \label{eq:intcriterion} 
	\int_{(1,\infty)} \log(q)\ \Prob_j\Big[\sup_{0<t\leq \tau(j)} \e^{\wh{\zeta}_{t-}}\, |\Delta (-\wh{L}^\sharp_t)|\in\,\diff q\Big]\ <\ \infty,
\end{equation}
where $\tau(j):=\inf\{t>0: J_t=j\neq J_{t-}\}$ is the first return time of $J_t$ to $j$, and 
\begin{equation} \label{eq-Lbigjumps} 
	\wh{L}^\sharp_t\ =\ \sum_{0<s\leq t} \Delta \wh{L}_s \mathds{1}_{|\Delta \wh{L}_s|>1} + \sum_{i,j\in\cS} \sum_{n: T_{n}\leq t} Z_{\wh{L},n}^{ij}
\end{equation}
is a finite-variation pure-jump process. 
By duality, we may replace $(\wh{\zeta}, \wh{L}^\sharp)$ in \eqref{eq:intcriterion} with $(-\zeta, -L^\sharp)$, where $L^\sharp$ is defined analogously to \eqref{eq-Lbigjumps}.
Further, from \eqref{eq-ULviaxieta},  
$$ \e^{-\zeta_{t-}} \Delta L_{t}\, =\, \e^{-\zeta_{t}} \Delta  \eta_{t}.$$
Thus, \eqref{eq:intcriterion} is equivalent to 
$$ \int_{(1,\infty)} \log(q)\, \Prob_j\bigg[\sup_{0<t\leq \tau(j)}\e^{-\zeta_{t}}\, |\Delta\eta^\sharp_t|\in\,\diff q\bigg]\ <\ \infty,$$
where $\eta^\sharp$ is a finite-variation pure-jump process defined analogously to \eqref{eq-Lbigjumps}. Now, by \cite[Cor.~3.9]{BDTS}, Assumption \textsf{(A3)} implies
\begin{gather*}
	\Erw\log^{+}|\eta_{1}^{(j)}|\,<\,\infty\quad\text{for all }j\in\cS
	\shortintertext{and}
	\Erw\log^{+}|Z^{ij}_{\eta}|\,<\,\infty\quad\text{for all }i,j\in\cS\text{ with }q_{ij}>0,
\end{gather*}
while \textsf{(A4)} ensures finite moments for $\zeta$. Hence $(V_{t})_{t\ge0}$ admits a unique stationary law. As this law coincides with $\nu$, the proof is complete. \hfill\qedsymbol

\section*{Appendix. The occupation measure formula}

In the following, we provide a version of the \emph{occupation measure formula} that is not the most general one, but is tailored to our needs.

\vspace{.1cm}
Consider a probability space $(\Omega,\cF,\Prob)$ together with a filtration $\cF=(\cF_{n})_{n\ge 0}$ and a Markov chain $(M_{n},\cF_{n})_{n\ge 0}$ with general state space $(\cX,\fX)$. As before, we write $\Prob_{x}$ for $\Prob(\,\cdot\,|\,M_{0}=x)$ and define $\Prob_{\lambda}=\int_{\cX}\Prob_{x}\,\lambda(dx)$ for any probability measure $\lambda$ on $(\cX,\fX)$.
Let $\sigma\ge 1$ be a stopping time with respect to $\cF$ and $\lambda$ a distribution on $\cX$ such that
$$ \sfm\,:=\,\Erw_{\lambda}\sigma\,<\,\infty\quad\text{and}\quad\Prob_{\lambda}(M_{\sigma}\in\cdot)\,=\,\lambda. $$
Then the normalized occupation measure
\begin{equation}\label{eq:normalized occupation measure formula}
	\pi(A)\ =\ \frac{1}{\sfm}\,\Erw_{\lambda}\Bigg[\sum_{n=1}^{\sigma}\1_{A}(M_{n})\Bigg],\quad A\in\fX
\end{equation}
defines a stationary distribution of $(M_{n})_{n\ge 0}$. Equivalently,
\begin{equation}\label{eq:occupation measure formula}
	\Erw_{\lambda}\Bigg[\sum_{n=1}^{\sigma}h(M_{n})\Bigg]\ =\ \sfm\,\Erw_{\lambda}h(M_{0})\ =\ \sfm\int h(x)\,\pi(dx)
\end{equation}
for any measurable $h:\cS\times\R\to\R$ such that at least one of 
$\Erw^{(\theta)}h^{-}(M_{0})$ or $\Erw^{(\theta)}h^{+}(M_{0})$ is finite. The formulas remain valid when summing over $n=0,\dots,\sigma-1$.

\vspace{.1cm}
For the \MMLIFS\ $(\xi_{n},R_{n})_{n\ge 0}$, which is stationary Markov chain under $\Prob_{\pi}$, we can choose $\sigma=\tau(i):=\inf\{n\ge 1:\xi_{n}=i\}$ for arbitrary $i\in\cS$. Then $\sfm=\sfm_{i}:=\Erw_{i}\tau(i)<\infty$ by the positive recurrence of $(\xi_{n})_{n\ge 0}$. Hence, formula \eqref{eq:normalized occupation measure formula} provides
\begin{gather}
	\Prob_{\pi}[(\xi_{0},R_{0})\in\cdot]\ =\ \frac{1}{\sfm_{i}}\,\Erw_{i}\Bigg[\sum_{n=0}^{\tau(i)-1}\1_{\{(\xi_{n},R_{n})\in\cdot\}}\Bigg],\label{eq2:normalized occupation measure formula}
	\shortintertext{in particular}
	\Prob_{\pi}[R_{0}\in\cdot]\ =\ \frac{1}{\sfm_{i}}\,\Erw_{i}\Bigg[\sum_{n=0}^{\tau(i)-1}\1_{\{R_{n}\in\cdot\}}\Bigg]\label{eq3:normalized occupation measure formula}
\end{gather}
for any $i\in\cS$.

%
%

\bibliographystyle{imsart-number}

\end{document}